\documentclass[12pt]{amsart}
\usepackage[top=1.5in, bottom=1.5in, left=1.4in, right=1.4in]{geometry}  
\usepackage{geometry}                
\usepackage{graphicx}
\usepackage{amssymb}
\usepackage{epstopdf}
\usepackage{color}
\usepackage{url}
\usepackage{verbatim}
\usepackage[lined,algonl,boxed,norelsize]{algorithm2e}
\title[Zonotopal bijections for regular matroids]{Geometric bijections for regular matroids, zonotopes, and Ehrhart theory}
\author{Spencer Backman, Matthew Baker, Chi Ho Yuen}
\date{\today}  

\newcommand{\QQ}{\mathbb{Q}}
\newcommand{\RR}{\mathbb{R}}
\newcommand{\ZZ}{\mathbb{Z}}

\numberwithin{equation}{section}
\theoremstyle{definition}

\newtheorem{theorem}{Theorem}[subsection]
\newtheorem{lemma}[theorem]{Lemma}

\newtheorem{corollary}[theorem]{Corollary}

\newtheorem{remark}[theorem]{Remark}
\newtheorem{proposition}[theorem]{Proposition}
\newtheorem{example}[theorem]{Example}

\newcommand{\Col}{\operatorname{Col}}
\newcommand{\Jac}{\operatorname{Jac}}
\newcommand{\Pic}{\operatorname{Pic}}

\newcommand{\sign}{\operatorname{sign}}
\newcommand{\coker}{\operatorname{coker}}

\DeclareRobustCommand{\rchi}{{\mathpalette\irchi\relax}}
\newcommand{\irchi}[2]{\raisebox{\depth}{$#1\chi$}}

\begin{document}

\begin{abstract}
Let $M$ be a {\em regular matroid}.  The {\em Jacobian group} ${\rm Jac}(M)$ of $M$ 
is a finite abelian group whose cardinality is equal to the number of {\em bases} of $M$. This group
generalizes the definition of the Jacobian group (also known as the critical group or sandpile group) $\Jac(G)$ of a graph $G$ (in which case bases of the corresponding regular matroid are spanning trees of $G$).  

There are many explicit combinatorial bijections in the literature between the Jacobian group of a graph ${\rm Jac}(G)$ and spanning trees.  However, most of the known bijections use {\em vertices} of $G$ in some essential way and are inherently ``non-matroidal''. In this paper, we construct a family of explicit and easy-to-describe bijections between the Jacobian group of a regular matroid $M$ and bases of $M$, many instances of which are new even in the case of graphs. We first describe our family of bijections in a purely combinatorial way in terms of orientations; more specifically, we prove that the Jacobian group of $M$ admits a canonical simply transitive action on the set ${\mathcal G}(M)$ of circuit-cocircuit reversal classes of $M$, and then define a family of combinatorial bijections $\beta_{\sigma,\sigma^*}$ between ${\mathcal G}(M)$ and bases of $M$.  (Here $\sigma$ (resp. $\sigma^*$) is an {\em acyclic signature} of the set of circuits (resp. cocircuits) of $M$.)
We then give a geometric interpretation of each such map $\beta=\beta_{\sigma,\sigma^*}$ in terms of zonotopal subdivisions which is used to verify that $\beta$ is indeed a bijection. 

Finally, we give a combinatorial interpretation of lattice points in the zonotope $Z$; by passing to dilations we obtain
a new derivation of Stanley's formula linking the Ehrhart polynomial of $Z$ to the Tutte polynomial of $M$. 
\end{abstract}

\maketitle

\section{Introduction}

\subsection{The main bijection in the case of graphs}

Let $G$ be a connected finite graph.
The {\em Jacobian group} ${\rm Jac}(G)$ of $G$ (also called the sandpile group, critical group, etc.) is a finite abelian group canonically associated to $G$ whose cardinality equals the number of spanning trees of $G$. Since in most cases there is no distinguished spanning tree to correspond to the identity element, there is no canonical bijection between $\Jac(G)$ and the set ${\mathcal T}(G)$ of spanning trees of $G$. However, many constructions of combinatorial bijections starting with some fixed additional data are known.  We mention, for example: the Cori--Le Borgne bijections that use an ordering of the edges as well as a fixed vertex \cite{cori2001burning}, Perkinson, Yang and Yu's bijections that use an ordering of the vertices \cite{perkinson2015dfs}, and Bernardi's bijections that use a cyclic ordering of the edges incident to each vertex \cite{bernardi2008tutte}.

In this paper we describe a new family of combinatorial bijections between $\Jac(G)$ and ${\mathcal T}(G)$.  Our bijections are very simple to state, though proving that they are indeed bijections is not so simple.  Another feature is that our bijections are formulated in a ``purely matroidal'' way, and in particular they generalize from graphs to {\em regular matroids}.  We will first state the main result of this paper in the language of graphs, and then give the generalization to regular matroids.

What we will in fact do is establish a family of bijections between ${\mathcal T}(G)$ and the set ${\mathcal G}(G)$ of {\em cycle-cocycle equivalence classes} of orientations of $G$. The latter was introduced by Gioan \cite{gioan2007enumerating,gioan2008circuit} and by definition is the set of equivalence classes of orientations of $G$ with respect to the equivalence relation generated by directed cycle reversals and directed cut reversals. We will write $[{\mathcal O}]$ to denote the equivalence class containing an orientation ${\mathcal O}$. ${\mathcal G}(G)$ is known to be a torsor for $\Jac(G)$ in a canonical way (i.e., there is a canonical simply transitive group action of $\Jac(G)$ on ${\mathcal G}(G)$) \cite{backman2014riemann}. By fixing a class in ${\mathcal G}(G)$ to correspond to the identity element of $\Jac(G)$, we then obtain a bijection between $\Jac(G)$ and ${\mathcal T}(G)$.

To state our main bijection for graphs, let ${\mathcal C}(G)$ (resp. ${\mathcal C}^*(G)$) denote the set of simple cycles (resp. minimal cuts, i.e., bonds) of $G$, and define a {\bf cycle signature} (resp. {\bf cut signature}) on $G$ to be a choice, for each $C \in {\mathcal C}(G)$ (resp. $C \in {\mathcal C}^*(G)$), of an orientation of $C$. By fixing an reference orientation for each edge, we can identify directed cycles (resp. directed cuts) with elements of $\mathbb{Z}^{E(G)}$. Now we call a cycle signature $\sigma$ (resp. cut signature $\sigma^*$) {\bf acyclic} if whenever $a_C$ are nonnegative reals with
\[
\sum_{C \in {\mathcal C}(G)} a_C \sigma(C) = 0
\]
in $\mathbb{Z}^{E(G)}$ 
(resp. $\sum_{C \in {\mathcal C}^*(G)} a_C \sigma^*(C) = 0$) we have $a_C = 0$ for all $C$.

\begin{example}\label{edgeorder}
Fix a total order $e_1 < \cdots < e_m$ and a reference orientation ${\mathcal O}$ of $E(G)$, and orient each simple cycle $C$ compatibly with the reference orientation of the smallest element in $C$.  This gives an acyclic signature of ${\mathcal C}(G)$. Indeed, suppose the signature is not acyclic and take some nontrivial expression $\sum_{C \in {\mathcal C}(G)} a_C \sigma(C) = 0$. Let $e$ be the minimum element appearing in some cycle in the support of this expression.  Then the element $e$ must be appear with different orientations in at least two different cycles, and thus one of these cycles is not oriented according to $\sigma$, a contradiction. One can, in an analogous way, use ${\mathcal O}$ to define an acyclic signature of ${\mathcal C}^*(G)$.
\end{example}

Recall that if $T$ is a spanning tree of $G$ and $e \not\in T$ (resp. $e \in T$), there is a unique cycle  $C(T,e)$ (resp. cut $C^*(T,e)$) contained in $T \cup \{ e \}$ (resp. containing $T \backslash \{ e \}$), called the {\em fundamental cycle} (resp. {\em fundamental cut}) associated to $T$ and $e$. With this notation in place, we can now state our main bijection in the case of graphs:

\begin{theorem} \label{thm:mainbijectionforgraphs}
Let $G$ be a connected finite graph, and fix acyclic signatures $\sigma$ and $\sigma^*$ of ${\mathcal C}(G)$ and ${\mathcal C}^*(G)$, respectively. Given a spanning tree $T \in {\mathcal T}(G)$, let ${\mathcal O}(T)$ be the orientation of $G$ in which we orient each $e \not\in T$ according to its orientation in $\sigma(C(T,e))$ and each $e \in T$ according to its orientation in $\sigma^*(C^*(T,e))$. Then the map $T \mapsto [{\mathcal O}(T)]$ is a bijection between ${\mathcal T}(G)$ and ${\mathcal G}(G)$.
\end{theorem}

The bijection in Theorem~\ref{thm:mainbijectionforgraphs} appears to be new even in the special case where $\sigma$ and $\sigma^*$ are defined as in Example~\ref{edgeorder}.

\begin{example}\label{planar}
Suppose that $G$ is a plane graph and define $\sigma$ by orienting each simple cycle of $G$ counterclockwise.  Similarly, define $\sigma^*$ by orienting each simple cycle of the dual graph $G^*$ clockwise and composing with the natural bijection between oriented cuts of $G$ and oriented cycles of $G^*$. By \cite[Theorem 15]{yuen2015geometric}, the simply transitive action of $\Jac(G)$ on ${\mathcal T}(G)$ afforded by Theorem~\ref{thm:mainbijectionforgraphs} in this case coincides with the ``Bernardi torsor'' defined in \cite{bakeryao2016torsor} and {\em a posteriori} with the ``rotor-routing torsor'' defined in \cite{chan2015rotor,chan2015duality}.  In particular, we get a new ``geometric'' proof of the bijectivity of the Bernardi map.

This example is in fact a special case of Example~\ref{edgeorder}.  Indeed, let $q^*$ be the vertex of $G^*$ corresponding to the unbounded face of $G$, and fix a spanning tree $T^*$ of $G^*$. Let $\mathcal{O}^*$ be any orientation of $G^*$ in which the edges of $T^*$ are oriented away from $q^*$, and fix any total order on $E(G^*)$ in which every edge of the rooted tree $T^*$ has a larger label than its ancestors while being smaller than all the edges outside $T^*$. Using the natural bijection\footnote{If $e$ and $e^*$ are dual edges of $G$ and $G^*$, respectively, then given an orientation for $e^*$ we orient $e$ by rotating the orientation of $e^*$ clockwise locally near the crossing of $e$ and $e^*$.} between oriented edges of $G$ and of $G^*$, this gives an orientation $\mathcal{O}$ of $G$ and a total order $<$ on $E(G)$.  Then the cycle signature $\sigma$ associated to $(\mathcal{O},<)$ by the rule in Example~\ref{edgeorder} will orient every simple cycle of $G$ counterclockwise.
\end{example}

\subsection{Generalization to regular matroids} \label{sec:introregular}

As mentioned previously, an interesting feature of the bijection given by Theorem~\ref{thm:mainbijectionforgraphs} is that it admits a direct generalization to {\em regular matroids}.

Regular matroids are a particularly well-behaved and widely studied class of matroids which contain graphic (and co-graphic) matroids as a special case. More precisely, a regular matroid can be thought of as an equivalence class of totally unimodular integer matrices. 
See \S\ref{sec:regularmatroids} for further details.

If $G$ is a graph, one can associate a regular matroid $M(G)$ to $G$ by taking the (modified) adjacency matrix of $G$.
By a theorem of Whitney, the equivalence class of $A$ determines the graph $G$ up to ``2-isomorphism'' (and in particular determines $G$ up to isomorphism if $G$ is assumed to be $3$-connected).

Let $M$ be a regular matroid. One can define the set ${\mathcal C}(M)$ of {\em signed circuits} of $M$ (resp. the set ${\mathcal C}^*(M)$ of {\em signed cocircuits} of $M$) in a way which generalizes the corresponding objects when $M=M(G)$.  Similarly, one has a set $B(M)$ of {\em bases} of $M$, generalizing the notion of spanning tree for graphs, and a set ${\mathcal G}(M)$ of cycle-cocycle equivalence classes generalizing the corresponding set for graphs.

In Section 4.3 of his Ph.D. thesis, Criel Merino defined the {\em critical group} (which we will call the {\em Jacobian}) ${\rm Jac}(M)$ of $M$, generalizing the critical group of a graph. 
By results of Merino and Gioan, the cardinalities of $\Jac(M)$, $B(M)$, and ${\mathcal G}(M)$ all coincide\footnote{The fact that these cardinalities are equal is essentially a translation of the natural extension of Kirchhoff's Matrix-Tree theorem to regular matroids \cite{maurer1976matrix},\cite[Theorem 4.3.2]{merino1999matroids}. A ``volume proof'' of the Matrix-Tree theorem for regular matroids based on zonotopal subdivisions is given in \cite{dall2014polyhedral}.  These authors do not consider the problem of giving explicit combinatorial bijections between bases of $M$ and the Jacobian group.}.

Generalizing the known case of graphs \cite{backman2014riemann}, we prove:

\begin{theorem} \label{thm:torsortheorem}
${\mathcal G}(M)$ is canonically a torsor for $\Jac(M)$.
\end{theorem}

In view of this result, in order to construct a bijection between elements of $\Jac(M)$ and bases of $M$, it suffices to give a bijection between $B(M)$ and ${\mathcal G}(M)$: 
if we fix an arbitrary element of ${\mathcal G}(M)$ (e.g. by fixing a reference orientation of $M$), our torsor induces a bijection between ${\rm Jac}(M)$ and ${\mathcal G}(M)$.
One can generalize the notion of acyclic signature and fundamental cycles (resp. cuts) in a straightforward way from graphs to regular matroids.
Theorem~\ref{thm:mainbijectionforgraphs} then admits the following generalization to regular matroids:

\begin{theorem} \label{thm:mainbijectionforregularmatroids}
Let $M$ be a regular matroid, and fix acyclic signatures $\sigma$ and $\sigma^*$ of ${\mathcal C}(M)$ and ${\mathcal C}^*(M)$, respectively. Given a basis $B \in B(M)$, let ${\mathcal O}(B)$ be the orientation of $M$ in which we orient each $e \not\in B$ according to its orientation in $\sigma(C(B,e))$ and each $e \in B$ according to its orientation in $\sigma^*(C^*(B,e))$. Then the map $B \mapsto [{\mathcal O}(B)]$ gives a bijection $\beta : B(M) \to {\mathcal G}(M)$. 
\end{theorem}

\medskip

Most known combinatorial bijections between elements of ${\rm Jac}(G)$ and spanning trees of a graph $G$ do not readily extend to the case of regular matroids, as they use vertices of the graph in an essential way.
The only other work we are aware of giving explicit bijections between elements of $\Jac(M)$ and bases of a regular matroid $M$ are the papers of Gioan and Gioan--Las Vergnas \cite{gioan2002correspond,gioan2005activity}\footnote{Technically speaking, Gioan and Las Vergnas do not produce a bijection between bases and elements of $\Jac(M)$; they produce a bijection between $B(M)$ and $\mathcal{X}(M;\sigma,\sigma^*)$, where $\sigma$ and $\sigma^*$ are determined by a total order on the edges and a reference orientation as in Example~\ref{edgeorder}; see \S\ref{sec:brief_overview} for the definition of $\mathcal{X}(M;\sigma,\sigma^*)$.} and the as-yet unpublished recent work of Shokrieh \cite{farbod2016draft}. Our family of combinatorial bijections appears to be quite different from those of Gioan--Las Vergnas. 

\subsection{Brief overview of the proof of the main combinatorial bijections} \label{sec:brief_overview}

Although the statement of Theorem~\ref{thm:mainbijectionforgraphs} and its generalization Theorem~\ref{thm:mainbijectionforregularmatroids} to regular matroids $M$ are completely combinatorial, we do not know any simple combinatorial proof.  Our proof involves the geometry of a zonotopal subdivision associated to a matrix $A$ representing $M$.

Concretely, fix a totally unimodular $r \times m$ matrix $A$ representing $M$, where $r$ is the rank of $A$, and the columns of $A$ are indexed by a set $E$ of cardinality $m$. Denote by $V^* \subseteq \RR^E$ the row space of $A$ and by $\pi_{V^*}$ the orthogonal projection from $\RR^E$ to $V^*$. Let $u_e \in \RR^E$ be the standard coordinate vector corresponding to $e \in E$.  The {\em column zonotope} $Z_A \subset \RR^r$ (resp. {\em row zonotope} $\widetilde{Z_A} \subset \RR^E$) associated to $A$ is defined to be the Minkowski sum of the columns of $A$ (resp. the Minkowski sum of the orthogonal projections $\pi_{V^*}(u_e)$ for $e \in E$).
The linear transformation $L : v \mapsto Av$ gives an isomorphism from $V^*$ to $\RR^r$ taking $\widetilde{Z_A}$ to $Z_A$ and preserving lattice points (cf. Lemma~\ref{lem:RC_zonotope_isomorphic}). 
The reason that we introduce two versions of essentially the same object is mostly for the sake of notational convenience.

An {\em orientation} $\mathcal{O}$ of $M$ is a function $E \to \{ -1, 1 \}$.
An orientation $\mathcal{O}$ is {\em compatible} with a signed circuit $C$ of $M$ if $\mathcal{O}(e) = C(e)$ for all $e$ in the support of $C$.

If $\mathcal{O}$ is an orientation and $C$ is a signed circuit compatible with $\mathcal{O}$, we can perform a {\em circuit reversal} taking $\mathcal{O}$ to the orientation $\mathcal{O}'$ 
defined by $\mathcal{O}'(e) = -\mathcal{O}(e)$ if $e$ is in the support of $C$ and $\mathcal{O}'(e) = \mathcal{O}(e)$ otherwise.

Let $\sigma$ be an acyclic signature of ${\mathcal C}(M)$.
We say that $\mathcal{O}$ is {\em $\sigma$-compatible} if every signed circuit $C$ of $M$ compatible with $\mathcal{O}$ is oriented according to $\sigma$.
By Proposition~\ref{prop:discretecompatible}, every circuit-reversal equivalence class of orientations contains a unique $\sigma$-compatible orientation.

The connection between $\sigma$-compatible orientations and the zonotopes defined above is given by the following result.
For the statement, given an orientation $\mathcal{O}$ of $M$ and $e \in E$, define $w_e \in \RR^r$ to be $0$ if $\mathcal{O}(e) = -1$ and 
to be the $e^{\rm th}$ column of $A$ if $\mathcal{O}(e) = 1$.
Define $\psi(\mathcal{O}) \in Z_A$ by

\begin{equation}
\psi(\mathcal{O}) := \sum_{e \in E} w_e \in Z_A.
\end{equation}
Then the map $\psi$ induces a bijection between circuit-reversal classes of orientations of $M$ and lattice points of the zonotope $Z_A$ (cf. Proposition~\ref{prop:latticepointprop}).

Fix a reference orientation $\mathcal{O}_0$ of $M$.
Each acyclic signature $\sigma$ of ${\mathcal C}(M)$ gives rise to a subdivision of $Z_A$ into smaller zonotopes $Z(B)$, one for each basis $B$ of $M$,
in the following way.

Let $B$ be a basis of $M$.  For each $e \not\in B$, define $v_e \in V^*$ to be $0$ if the reference orientation of $e$ coincides with the orientation of $e$
in $\sigma(C(B,e))$, and to be the $e^{\rm th}$ column of $A$ otherwise.
Define 
\[
Z(B) := \sum_{e \in B} [0,A_e] + \sum_{e \not\in B} v_e \subseteq Z_A \subset \RR^r.
\]

By Proposition~\ref{prop:zonotopedecomp}, the collection of $Z(B)$'s gives a {\em zonotopal subdivision} (also known in the literature as a {\em tiling}) $\Sigma$ of $Z_A$. Similarly, via the map $L$, the various $\widetilde{Z(B)}:=L^{-1}(Z(B))$'s give a zonotopal subdivision $\widetilde{\Sigma}$ of $\widetilde{Z_A}$.

\medskip

We now explain briefly how these results are used to prove 
Theorem~\ref{thm:mainbijectionforregularmatroids}.  

Let $\sigma,\sigma^*$ be acyclic signatures of $C(M)$ and $C^*(M)$, respectively. An orientation is called {\em $(\sigma,\sigma^*)$-compatible} if it is both $\sigma$-compatible and $\sigma^*$-compatible, and we denote the set of such orientations by $\mathcal{X}(M;\sigma,\sigma^*)$.

\begin{theorem} \label{thm:SScompatibleorientation}
Let $\hat{\beta}$ be the map which sends a basis $B$ to the orientation $\mathcal{O}(B)$ defined in Theorem~\ref{thm:mainbijectionforregularmatroids}. Let $\chi$ be the map which sends an orientation $\mathcal{O}$ to its circuit-cocircuit reversal class $[\mathcal{O}]$, so that  
 $\beta=\chi\circ\hat{\beta}$.  
\begin{enumerate}
\item The image of $\hat{\beta}$ is contained in $\mathcal{X}(M;\sigma,\sigma^*)$, and $\hat{\beta}$ gives a bijection between $B(M)$ and $\mathcal{X}(M;\sigma,\sigma^*)$.
\item The map $\chi$ restricted to $\mathcal{X}(M;\sigma,\sigma^*)$ induces a bijection between $\mathcal{X}(M;\sigma,\sigma^*)$ and $\mathcal{G}(M)$. 
\end{enumerate}
\end{theorem}

\begin{remark}
The proofs of Theorem~\ref{thm:torsortheorem} and Theorem~\ref{thm:mainbijectionforregularmatroids} do not assume {\em a priori} that  $|B(M)|=|\mathcal{X}(M;\sigma,\sigma^*)|=|\mathcal{G}(M)|=|\Jac(M)|$ for a regular matroid $M$, thus our work provides an independent proof of these equalities. Furthermore, we will show in Theorem~\ref{thm:realizablemainthm} below that the equality $|B(M)|=|\mathcal{X}(M;\sigma,\sigma^*)|$ continues to hold under the weaker assumption that $M$ is realizable over $\mathbb{R}$.
\end{remark}

\medskip

By Lemma~\ref{lem:gordan}, we may choose a vector $w' \in V^*$ which is compatible with $\sigma^*$, in the sense that $w' \cdot \sigma^*(C) > 0$ for each cocircuit $C$ of $M$.
Note that the zonotopal subdivision $\widetilde{\Sigma}$ of $\widetilde{Z_A}$ depends only on $\sigma$ (and the reference orientation $\mathcal{O}_0$) and the vector $w'$ depends only on $\sigma^*$.

The following theorem shows that the combinatorially defined map $\hat{\beta} : {\mathcal B}(M) \to {\mathcal X}(M)$ can be interpreted geometrically as first identifying a basis with a maximal cell in our zonotopal subdivision and then applying a ``shifting map''.

\medskip

\begin{theorem} \label{thm:betaphiagree}
\begin{enumerate}
\item Let $B$ be a basis of $M$. For all sufficiently small $\epsilon > 0$ the image of $\widetilde{Z(B)}$ under the map $v \mapsto v + \epsilon w'$ contains a unique lattice point $\widetilde{z_B}$ of $\widetilde{Z_A}$, which corresponds to a unique $(\sigma,\sigma^*)$-compatible discrete orientation $\mathcal{O}'_B$.
\item The map $\phi$ which takes each basis $B$ to the orientation $\mathcal{O}'_B$ coincides with the map $\hat{\beta}$ appearing in the statement of Theorem~\ref{thm:mainbijectionforregularmatroids}, and hence $\hat{\beta}$ gives a bijection between $B(M)$ and $\mathcal{X}(M;\sigma,\sigma^*)$.
\end{enumerate}
\end{theorem}

Theorem~\ref{thm:mainbijectionforregularmatroids} is a simple consequence of Theorem \ref{thm:SScompatibleorientation} and Theorem \ref{thm:betaphiagree}. 

\medskip

\begin{example}
Let $G$ be a graph, and fix a vertex $q$ of $G$. In \cite{an2014canonical}, the authors prove that the {\em break divisors} of $G$ are the divisors associated to $q$-connected orientations offset by a chip at $q$. In other words (in the notation of \cite[Lemma 3.3]{an2014canonical}), a divisor $D$ is a break divisor if and only if $D = (q) + \nu_{\mathcal O}$ for some $q$-connected orientation $\mathcal O$.   They also show that break divisors of the corresponding metric graph $\Gamma$ induce a canonical subdivision of the $g$-dimensional torus $\Pic^g(\Gamma)$ into paralleletopes indexed by spanning trees of $G$, with the vertices of the subdivision corresponding to the break divisors of $G$. By applying a small generic shift to the vertices, this yields a family of ``geometric bijections'' between break divisors and spanning trees (cf.~\cite[Remark 4.26]{an2014canonical}).

We claim that the geometric bijections defined in \cite{an2014canonical} can be thought of as special cases of the bijections afforded by Theorem~\ref{thm:mainbijectionforregularmatroids}. 
By \cite[Theorem 10]{yuen2015geometric}, each such geometric bijection gives rise in a natural way to an acyclic orientation $\sigma$ of the cycles of $G$. 
On the other hand, we can use the recipe described in Example~\ref{planar} to produce a cut signature $\sigma^*$ such that every cut is oriented away from $q$.
Given a spanning tree $T$, the orientation $\mathcal{O}_T$ associated to the pair $(\sigma,\sigma^*)$ by Theorem~\ref{thm:mainbijectionforregularmatroids} will have the property that every edge $e$ in $T$ (considered as a tree rooted at $q$) is oriented away from $q$, and therefore $\mathcal{O}_T$ is $q$-connected \cite[Section 3]{backman2014partial}.  Let $D_T = \nu_{\mathcal{O}_T} + (q)$ be the corresponding break divisor.
Then $T \mapsto D_T$ will be the geometric bijection we started with.
\end{example}

\subsection{A partial extension to matroids realizable over $\mathbb{R}$} \label{sec:realizablecase}

Although the equality $|B(M)|=|\mathcal{G}(M)|=|\Jac(M)|$ does not hold for general oriented matroids (indeed, $\Jac(M)$ is not even well-defined in the general case), the notions of acyclic circuit/cocircuit signatures and $(\sigma,\sigma^*)$-compatible orientations continue to make sense whenever $M$ is realizable over ${\mathbb{R}}$. Furthermore, the geometric setup used to prove Theorem~\ref{thm:betaphiagree}, as well as the first half of Theorem~\ref{thm:SScompatibleorientation}, does not require $M$ to be regular but only realizable. Therefore we have the following result, which will be proved in \S\ref{sec:proof}:

\begin{theorem} \label{thm:realizablemainthm}
Let $M$ be an oriented matroid which is realizable over $\mathbb{R}$, and let $\sigma,\sigma^*$ be acyclic signatures of $C(M),C^*(M)$, respectively. Then the map $\hat{\beta}:B(M)\rightarrow\mathcal{X}(M;\sigma,\sigma^*)$ is a bijection.
\end{theorem}

An ingredient in the proof of Theorem~\ref{thm:realizablemainthm} is a continuous analogue of orientations (which we will refer to as {\em discrete orientations} whenever there is a risk of confusion), hence continuous analogues of acyclic signatures and circuit-cocircuit reversal systems. Using these notions, we can provide a combinatorial interpretation of {\em all} points of the zonotope $Z_A$ (not just the lattice points), thereby giving an alternate description of the zonotopal subdivision $\Sigma$ (resp. $\widetilde{\Sigma}$) of $Z_A$ (resp. $\widetilde{Z_A}$), which was defined above.

\subsection{Random sampling of bases} 

As in \cite{baker2013chip}, any computable bijection between bases and elements of $\Jac(M)$ gives rise an algorithm for randomly sampling bases of $M$. The idea is simple: it is easy to uniformly sample random elements from $\Jac(M)$, and applying the bijection yields a random spanning tree.

In order to make this into a practical method, one needs efficient algorithms for computing the basis associated to an element of $\Jac(M)$. In \S\ref{sec:bijectioncomputable} and Proposition~\ref{prop:groupactioncomputable}, we provide polynomial-time algorithms for such a task with respect to the family of bijections given by Theorem~\ref{thm:torsortheorem} and Theorem~\ref{thm:mainbijectionforregularmatroids}. We note that while our map from $\Jac(M)$ to $B(M)$ has a strong combinatorial flavor, our inverse algorithm uses ideas from linear programming.

We also remark that there are other algorithms for sampling random bases of a regular matroid, such as the random walk based method by Dyer and Frieze \cite{dyer1994randomwalks} (whose analysis also makes use of zonotopes). Since our algorithm requires solving multiple linear programs, its runtime is probably slower than some other known algorithms. However, our method can generate an exact uniform distribution using an information theoretical minimum amount of randomness, cf. the discussion by Lipton \cite{liptonblog}.

\subsection{Connections to Ehrhart theory and the Tutte polynomial}

Every matroid $M$ of rank $r$ has an associated {\em Tutte polynomial} $T_M(x,y)$, and every lattice polytope $P$ (e.g. the zonotope $Z_A$) has an associated {\em Ehrhart polynomial} $E_P(q)$ which counts the number of lattice points in positive integer dilates of $P$.  Using the relationship between $Z_A$ and $\sigma$-compatible (discrete or continuous) orientations of $M$, we obtain a new proof of the following identity originally due to Stanley:
\begin{equation} \label{eq:stanley}
E_Z(q) = q^r T_{M}(1 + 1/q,1).
\end{equation}
The proof involves defining a ``dilation'' $qM$ of $M$ for each positive integer $q$, with associated zonotope $qZ_A$. 

We also describe a direct bijective proof (without appealing to Ehrhart reciprocity) of the fact that the number of interior lattice points in $qZ_A$ is 
\[
q^r T_{M}(1-1/q,1).
\]

\subsection{Related literature} \label{sec:literature}

The study of zonotopal tilings, i.e. tilings of a zonotope by smaller zonotopes, is a classical topic in the theory of oriented matroids first initiated by Shephard \cite{shephard1974zonotopes}.   The central theorem in this area is the Bohne-Dress Theorem \cite{bohne1992zonotopaler,dress1989oriented}, which states that the poset of zonotopal tilings ordered by refinement is isomorphic to the poset of single-element lifts of the associated oriented matroid.  For instance, it can be shown that the subdivisions we consider in the paper correspond to precisely the generic, realizable single-element lifts of realizable oriented matroids; we will not further elaborate on this connection as it does not play a significant role in this paper.


\section{Background}

\subsection{Regular matroids} \label{sec:regularmatroids}

In this section we recall the definition of regular matroids and related objects.
We assume that the reader is familiar with the basic theory of matroids; some standard references include the book on matroids by Oxley \cite{oxley2006matroid} and the book on oriented matroids by Bj{\"o}rner et~al.~\cite{bjorner1999oriented}.

\medskip

An $r \times m$ matrix $A$ of rank $r$ with integer entries is called {\em totally unimodular} if every $k \times k$ submatrix has determinant in $\{ 0, \pm 1 \}$ for all $1 \leq k \leq r$. A matroid is called {\em regular} if it is representable over $\QQ$ by a totally unimodular matrix.

\medskip

The following lemma (see \cite{su2010lattice}) is the key fact used to show that various definitions in the subject are independent of the choice of a totally unimodular matrix $A$ representing $M$:

\begin{lemma} \label{lem:transformation}
If $A,A'$ are totally unimodular $r \times m$ matrices representing $M$, one can transform $A$ into $A'$ by multiplying on the left by an $r \times r$ unimodular matrix $U$, then permuting columns or multiplying columns by $-1$.  
\end{lemma}

If $M$ is a regular matroid of rank $r$ on $E$ and $A$ is any $r \times m$ totally unimodular matrix representing $M$ over $\QQ$, we define
$\Lambda_A(M) := \ker(A) \cap \ZZ^E$.  By Lemma~\ref{lem:transformation}, the isometry class of this lattice depends only on $M$, and not on the choice of the matrix $A$.  It is denoted by $\Lambda(M)$ and called the {\em circuit lattice} of $M$.

Similarly, we define $\Lambda^*_A(M)$ to be the intersection of the row space of $A$ with $\ZZ^E$, or equivalently the $\ZZ$-span of the rows of $A$.
The isometry class of this lattice also depends only on $M$.  It is denoted by $\Lambda^*(M)$ and called the {\em cocircuit lattice} of $M$.
(For proofs of all these statements, see \cite[\S{4.3}]{merino1999matroids} or \cite[\S{2.3}]{su2010lattice})

\medskip

The Jacobian group $\Jac(M)$ is defined to be the determinant group of $\Lambda(M)$, i.e., $\Jac(M) = \Lambda(M)^\# / \Lambda(M)$ where $\Lambda^\#$ is the {\em dual lattice} of $\Lambda$, i.e.,
\[
\Lambda^\# = \{ x \in \Lambda \otimes \QQ \; : \; \langle x,y \rangle \in \ZZ \; \forall \; y \in \Lambda \}.
\]

There are canonical isomorphisms 
\begin{equation} \label{eq:canonisom}
\Lambda(M)^\# / \Lambda(M) \cong \Lambda^*(M)^\# / \Lambda^*(M) \cong \frac{\ZZ^E}{\Lambda_A(M) \oplus \Lambda_A^*(M)}
\end{equation}
for every totally unimodular matrix $A$ representing $M$ (cf. \cite[Lemma~1 of \S{4}]{bacher1997lattice}).

The order of $\Jac(M)$ is equal to the number of {\em bases} of the matroid $M$ (cf. \cite[Theorem 4.3.2]{merino1999matroids}).
Moreover, we have $|\Jac(M)| = |{\rm det}(A^T A)|$ (cf. \cite[p.317]{godsil2013algebraic}), and in fact $\Jac(M)$ can naturally be identified with the cokernel of $A^T A$:

\begin{proposition}
The map $\frac{\ZZ^E}{\Lambda_A(M) \oplus \Lambda_A^*(M)}\rightarrow\coker(AA^T)$ given by $[\gamma]\mapsto [A\gamma]$ is well-defined and is an isomorphism.
\end{proposition}

\begin{proof}
The map is well-defined because $A(\Lambda_A(M) \oplus \Lambda_A^*(M))=A(\Lambda_A^*(M))=A(\Col_{\mathbb{Z}}A^T)=\Col_{\mathbb{Z}}AA^T$, the equality also shows the map is injective. It is surjective because $Ax=b$ has a solution in $\ZZ^E$ for every $b\in\mathbb{Z}^r$, using the unimodularity of $A$.
\end{proof} 

\medskip

We now discuss regular matroids from an oriented matroid point of view. By \cite[Corollary 7.9.4]{bjorner1999oriented}, every oriented matroid structure on a regular matroid is realizable by some totally unimodular matrix, hence any two such structures differ by reorientations.

Let $C(\underline{M})$ (resp. $C^*(\underline{M})$) be the set of circuits (resp. cocircuits) of $M$.  
Let $A$ be any $r \times m$ totally unimodular matrix representing $M$ over $\QQ$.
An element $\alpha \in \Lambda_A(M)$ (resp. $\Lambda_A^*(M)$) is called a {\em signed circuit} (resp. {\em signed cocircuit}) of $M$ if $\alpha \neq 0$,
all coordinates of $\alpha$ are in $\{ 0, \pm 1 \}$, and the support of $\alpha$ is a circuit (resp. cocircuit) of $M$.
We let $C_A(M)$ (resp. $C^*_A(M)$) denote the set of signed circuits (resp. signed cocircuits) of $M$.
The notion of signed circuit (resp. signed cocircuit) is in fact intrinsic to $M$, independent of the choice of $A$, and thus it makes sense to speak of $C(M)$ and $C^*(M)$
as subsets of $\Lambda(M)$ and $\Lambda^*(M)$, respectively (cf. \cite[Lemma~10 and Proposition~12]{su2010lattice} and \cite[Theorem~4.3.4]{merino1999matroids}).

\medskip

There is a natural map $C(M) \to C(\underline{M})$ taking $\alpha \in C_A(M)$ to its support (with respect to any choice of $A$).
This map induces a bijection $C(M) / \langle \pm 1 \rangle \to C(\underline{M})$, i.e, for every circuit $\underline{C}$ of $M$ there are precisely two signed
circuits $\pm C$ with ${\rm supp}(C) = \underline{C}$.
(cf. \cite[Lemma 8]{su2010lattice} and \cite[Theorem~4.3.5]{merino1999matroids})
The same holds for cocircuits.

\subsection{Equivalence classes of orientations, and signatures}
\label{sec:equivclassorient}

An {\em (discrete) orientation} of a regular matroid $M$ is a map from the ground set $E$ of $M$ to $\{ -1, 1 \}$.
An orientation ${\mathcal O}$ is {\em compatible} with a signed circuit $C$ of $M$ if ${\mathcal O}(e) = C(e)$ for all $e$ in the support of $C$.

The {\em circuit reversal system} is the equivalence relation on the set $O(M)$ of all orientations of $M$ generated by {\em circuit reversals},
in which we reverse the sign of ${\mathcal O}(e)$ for all $e$ in (the support of) some signed circuit $C$ compatible with ${\mathcal O}$.
We can make the same definitions for cocircuits by replacing $M$ with its dual.

The {\em circuit-cocircuit reversal system} is the equivalence relation generated by both circuit and cocircuit reversals.  It is a theorem of Gioan \cite[Theorem 10(v)]{gioan2008circuit}(originally proved by a deletion-contraction argument) that the number of circuit-cocircuit equivalence classes of orientations is equal to the number of bases of $M$; Theorem~\ref{thm:torsortheorem} gives a {\em bijective} proof of this fact.

\medskip

A {\em signature} of $C(\underline{M})$ is a map $\sigma : C(\underline{M}) \to C(M)$ that picks an orientation for each (unsigned) circuit of the matroid underlying $M$.

A signature $\sigma$ of $C(\underline{M})$ is called {\em acyclic} if the only solution to $\sum_{C_i \in C(\underline{M})} a_i \sigma(C_i) = 0$ with the $a_i$ non-negative numbers is the trivial solution where all $a_i$ are equal to zero.
Signatures for $C^*(\underline{M})$ are defined analogously.

\section{Matroids over $\mathbb{R}$ and the main combinatorial bijection}\label{sec:proof}

Throughout this section, $M$ will denote an oriented matroid which is realizable over $\mathbb{R}$.  
Note in particular that every regular matroid has this property.

\subsection{Continuous circuit reversals and the zonotope associated to a representation of $M$}
\label{sec:continuousreversals}

The main goal of this section is to prove Theorem~\ref{thm:realizablemainthm}, which (when specialized to regular matroids) is a major component in the proof of Theorem~\ref{thm:mainbijectionforregularmatroids}. 

Our proof is geometric. In order to explain the basic idea, we fix once and for all a real $r \times m$ matrix $A$ realizing $M$, where $r$ is the rank of $M$ and the columns of $A$ are indexed by the elements of the ground set $E$ of $M$ (which we sometimes regard as $\{1,2,\ldots,m\}$).

\medskip

We first briefly explain how certain important notions that we introduced for regular matroids extend more generally to matroids representable over $\mathbb{R}$. 

For every circuit $C$ of $M$, the elements in $\ker(A)$ whose support is $C$, together with the zero vector, form a one-dimensional subspace $U_C$ in $\ker(A)$. 
Conversely, the support of any support-minimal nonzero element of $\ker(A)$ corresponds to a circuit of $M$. 
The two rays of $U_C$ correspond to the two orientations of $C$. Hence we may identify a signed circuit $C$ with an arbitrary vector ${\bf v}_C$ in the ray.
The same holds for cocircuits if we replace $\ker(A)$ by the row space of $A$.

The definition of an acyclic circuit (resp.~cocircuit) signature follows verbatim from the discussion in \S\ref{sec:equivclassorient}, except that we now have the equation  $\sum_{C_i \in C(\underline{M})} a_i {\bf v}_{\sigma(C_i)} = 0$, which is well-defined as different choices of ${\bf v}_{\sigma(C_i)}$ differ by a positive scalar multiple.

As a simple consequence of {\em Gordan's alternative} in the theory of linear programming \cite[p.~478]{bjorner1999oriented}, we have the following criterion/alternative description of an acyclic signature.

\begin{lemma} \label{lem:gordan}
Let $\sigma$ be a signature of $C(M)$. Then $\sigma$ is acyclic if and only if there exists $w \in {\mathbb R}^E$ such that $w\cdot{\bf v}_{\sigma(C)}>0$ for each circuit $C$ of $M$.
\end{lemma}

In the situation of Lemma~\ref{lem:gordan}, we say that $w$ {\em induces} $\sigma$. 
By the orthogonality of vectors representing signed circuits and cocircuits, given any pair of acyclic signatures $\sigma,\sigma^*$ of $C(M)$ and $C^*(M)$, respectively, there exists $w \in {\mathbb R}^E$ that induces both $\sigma$ and $\sigma^*$.

\medskip

We state an elementary lemma for realizable oriented matroids. In abstract oriented matroid terms, it is to say that every signed vector of an oriented matroid is a {\em conformal composition} of signed circuits.

\begin{lemma} \cite[Lemma 6.7]{Ziegler_book} \label{lem:circuitdecomposition}
Let $u\in\mathbb{R}^E$ be a vector in $\ker(A)$. Then $u$ can be written as a sum of signed circuits $\sum {\bf v}_C$ where the support of each $C$ is inside the support of $u$, and for each $e$ in the support of $C$, the signs of $e$ in $C$ and $u$ agree.
\end{lemma}

A {\em continuous orientation} ${\mathcal O}$ of $M$ is a function $E \to [-1,1]$, and the {\em $e$-th coordinate} of $\mathcal{O}$ is the value $\mathcal{O}(e)$.  If ${\mathcal O}(e) \in \{-1,1\}$ for all $e \in M$, we say that ${\mathcal O}$ is a {\em discrete orientation}.

A continuous orientation ${\mathcal O}$ is {\em compatible} with a signed circuit $C$ of $M$ if ${\mathcal O}(e) \neq -\sign(C(e))$ for all $e$ in the support of $C$. 
Given a continuous orientation ${\mathcal O}$ compatible with a signed circuit $C$, a {\em continuous circuit reversal} with respect to $C$ replaces ${\mathcal O}$ by a new continuous orientation ${\mathcal O}-\epsilon {\bf v}_C$ for some $\epsilon > 0$.  
(In particular, we require $\epsilon$ to be small enough so that $({\mathcal O}-\epsilon {\bf v}_C)(e) \in [-1,1]$ for all $e \in E$.)

The {\em continuous circuit reversal system} is the equivalence relation on the set $CO(M)$ of all continuous orientations of $M$ generated by all possible continuous circuit reversals. We can make the same definitions for cocircuits by replacing $M$ with its dual.

\medskip

Next we define the {\em (column) zonotope $Z_A$ associated to $A$} to be the Minkowski sum of the columns of $A$ (thought of as line segments in ${\mathbb R}^r$), i.e.,
\[
Z_A = \{ \sum_{i=1}^m c_i v_i \; : \; 0 \leq c_i \leq 1 \}
\]
where $v_1,\ldots,v_r$ are the columns of $A$.\footnote{Some authors consider variations on this zonotope, e.g. $\sum_{i=1}^m[-v_i,v_i]$, $\sum_{i=1}^m[-v_i/2,v_i/2]$, or $\sum_{i=1}^m[v_i^-,v_i^+]$, where $v^-$ and $v^+$ are the negative and positive parts of $v$, respectively.}

\begin{remark} \label{rmk:graphicmatrix}
When $M=M(G)$ is a graphic matroid, it is usually more convenient to take $A$ to be the full adjacency matrix of $G$, rather than a modified adjacency matrix with one row removed, when defining the corresponding zonotope.  This has the advantage of producing a canonically defined object, and since all of these different zonotopes are isomorphic, there is little harm in doing this.
\end{remark}

\medskip

There are several important connections between the zonotope $Z_A$ and equivalence classes of orientations of $M$.
For the statement, given $\alpha \in [-1,1]$ we denote by $\hat\alpha$ the real number $\frac{1}{2}(\alpha + 1) \in [0,1]$.
Define $\psi : O(M) \to Z_A$ be the map taking an orientation ${\mathcal O}$ (thought of as an element of $[-1,1]^E$) to 
\begin{equation} \label{eq:psidef}
\psi({\mathcal O}) := \sum_{i=1}^m \widehat{{\mathcal O}(e_i)} v_i \in Z_A.
\end{equation}

\medskip

\begin{proposition} \label{prop:ctslatticepointprop}
The map $\psi$ gives a bijection between continuous circuit-reversal classes of continuous orientations of $M$ and points of the zonotope $Z_A$.
\end{proposition}

\begin{proof} By definition, $\psi$ sends every continuous orientation to some point in $Z_A$, and $\psi$ is surjective. By the orthogonality of circuits and cocircuits, two continuous orientations in the same circuit-reversal class map to the same point of $Z_A$, so it remains to show the converse. Suppose $\psi(\mathcal{O})=\psi(\mathcal{O}')$. By Lemma~\ref{lem:circuitdecomposition}, $\mathcal{O}-\mathcal{O}'$ can be written as a sum of signed circuits in which each signed circuit is compatible with $\mathcal{O}$, and $\mathcal{O}$ can be transformed to $\mathcal{O}'$ via the corresponding continuous circuit reversals in any order.
\end{proof}

\subsection{Distinguished orientations within each equivalence class} \label{sec:DistOrient}

If we fix an acyclic signature $\sigma$ of $C(M)$, there is a natural way to pick out a distinguished continuous orientation from each continuous circuit reversal class.

Define a continuous orientation ${\mathcal O}$ to be {\em $\sigma$-compatible} if every signed circuit $C$ of $M$ compatible with ${\mathcal O}$ is oriented according to $\sigma$.  

\medskip

\begin{proposition} \label{prop:continuouscompatible}
Let $\sigma$ be an acyclic signature of $C(M)$. Then each continuous circuit-reversal class $M$ contains a unique $\sigma$-compatible continuous orientation.
\end{proposition}

\begin{proof} By Lemma~\ref{lem:gordan}, there exists $w\in\mathbb{R}^E$ such that $w\cdot{\bf v}_{\sigma(C)}>0$ for every circuit $C$ of $M$. Consider the function $P(\mathcal{O}'):=w\cdot\mathcal{O}'$. If $-\sigma(C)$ is compatible with $\mathcal{O}$ for some circuit $C$, then performing a continuous circuit reversal with $-\sigma(C)$ strictly increases the value of $P$, so every maximizer of $P$ inside a class (if exists) must be $\sigma$-compatible. The set of continuous orientations in a continuous circuit reversal class is the fiber of $\psi$ over a point in $Z_A$, which is a closed subset of the hypercube, so such maximizer must exist as $P$ is continuous.

For uniqueness, suppose there are two distinct $\sigma$-compatible continuous orientations $\mathcal{O},\mathcal{O}'$ in a continuous circuit-reversal class. By Lemma~\ref{lem:circuitdecomposition}, $\mathcal{O}$ can be transformed to $\mathcal{O}'$ via a series of continuous circuit-reversals in which each signed circuit involved is compatible with $\mathcal{O}$, hence agrees with $\sigma$. If the last signed circuit involved in the series of reversals is $C$, then $-C$ is a signed circuit compatible with $\mathcal{O}'$, hence it agrees with $\sigma$ as well, which is a contradiction.
\end{proof}

\begin{remark}
By interpreting $\sigma$-compatible orientations as maximizers of the linear function $P$, it is easy to see that the map $\mu:Z_A\rightarrow CO(M)$, which takes a point $z$ of $Z_A$ to the unique $\sigma$-compatible continuous orientation in the continuous circuit-reversal class corresponding to $z$, is a continuous section to the map $\psi$. Such a point of view is closely related to the classical theory of zonotopal tilings.
\end{remark}

We will call orientations that are compatible with both $\sigma$ and $\sigma^*$ {\em $(\sigma,\sigma^*)$-compatible orientations}. 

\medskip

The set of {\em discrete} $(\sigma,\sigma^*)$-compatible orientations will be denoted by $\mathcal{X}(M;\sigma,\sigma^*)$.
In \S\ref{discretestuffforregularmatroids}, we will establish an analogue of Proposition~\ref{prop:continuouscompatible} for discrete orientations of {\em regular} matroids.

\subsection{Bi-orientations and bases}

Let ${\mathcal O}$ be a continuous orientation of $M$.  We call an element $e \in E$ {\em bi-oriented} with respect to ${\mathcal O}$ if 
${\mathcal O}(e) \in (-1,1)$.

Note that if we {\em orient} any bi-oriented element $e$ in a $\sigma$-compatible continuous orientation $\mathcal{O}$, i.e., we set $\mathcal{O}(e)$ equal to either 1 or $-1$, the new continuous orientation is still $\sigma$-compatible.

\begin{proposition} \label{prop:uniqueext}
Let $\sigma$ be an acyclic signature of $C(M)$.
\begin{enumerate}
\item If ${\mathcal O}$ is a $\sigma$-compatible continuous orientation then the set of $e \in E$ which are bi-oriented with respect to ${\mathcal O}$ is independent in $M$.
\item If $B$ is a basis for $M$ and $b : B \to (-1,1)$ is any function, there is a unique $\sigma$-compatible continuous orientation 
${\mathcal O} = {\mathcal O}(B,b)$ such that ${\mathcal O}(e) = b(e)$ for all $e \in B$ and ${\mathcal O}(e) \in \{ \pm 1 \}$ for all $e \not\in B$.
\end{enumerate}
\end{proposition}

\begin{proof} For the first part, suppose the set $S$ of bi-oriented elements in a continuous orientation $\mathcal{O}$ is not independent. Then $S$ contains some circuit $C$, and $\mathcal{O}$ is compatible with both orientations of $C$, so $\mathcal{O}$ is not $\sigma$-compatible.

\medskip

For the uniqueness assertion in (2), note that each element not in $B$ must be oriented in agreement with the orientation of its fundamental circuit given by $\sigma$, as for otherwise the fundamental circuit will be compatible with $-\sigma$. Such unique choice of orientations outside $B$, together with $b$ itself, gives a continuous orientation $\mathcal{O}$.

\medskip

Now we claim that such $\mathcal{O}$ is $\sigma$-compatible. If not, then $\mathcal{O}$ is compatible with $-\sigma(C)$ for some circuit $C$. We choose $C$ such that $|C\setminus B|\neq 0$ is minimum among all such circuits. Pick any $e\in C\setminus B$ and let $C'$ be its fundamental circuit with respect to $B$. Then $\mathcal{O}$ is compatible with $\sigma(C')$ by construction. Pick suitable ${\bf v}_{-\sigma(C)},{\bf v}_{\sigma(C')}$ such that they agree on the $e$-th coordinate. Using Lemma~\ref{lem:circuitdecomposition}, we write ${\bf v}_{-\sigma(C)}-{\bf v}_{\sigma(C')}=\sum_{D} {\bf v}_D$ with $D$'s being signed circuits conformal with the left hand side, hence they do not contain $e$. Since ${\bf v}_{\sigma(C)}+{\bf v}_{\sigma(C')}+\sum_{D} {\bf v}_D=0$, at least one such $D$ is oriented opposite to $\sigma$ by acyclicity. Then such $D$ is compatible with $\mathcal{O}$: each element of $D$ is either in $B$ (which is bi-oriented in $\mathcal{O}$), or from $C$ and oriented as in $-\sigma(C)$ (which is compatible with $\mathcal{O}$). However, $D\setminus B\subset (C\setminus B)\setminus e$, contradicting the minimality of $C$.
\end{proof}

\subsection{Polyhedral subdivision of the zonotope}
\label{sec:polyhedralsubdivision}

Let $\sigma$ be an acyclic signature of $C(M)$.
For each basis $B$ of $M$, let $CO^\circ(B)$ be the set of $\sigma$-compatible continuous orientations of the form ${\mathcal O}(B,b)$ as $b$ ranges over all possible $b : B \to (-1,1)$.  Let $Z^\circ(B) = \psi(CO^\circ(B))$ be the projection of $CO^\circ(B)$ to $Z_A$, and let $Z(B)$ be the topological closure of $Z^\circ(B)$ in $Z_A$.  
Finally, let $CO(B) = \mu(Z(B))$ be the closure of $CO^\circ(B)$ in $CO(M)$.

\begin{proposition} \label{prop:zonotopedecomp}
\begin{enumerate}
\item The union of $Z(B)$ over all bases $B$ of $M$ is equal to $Z_A$, and if $B,B'$ are distinct bases then $Z^\circ(B)$ and $Z^\circ(B')$ are disjoint.
\item The collection of $Z(B)$ as $B$ varies over all bases $B$ for $M$ gives a polyhedral subdivision of $Z_A$ whose vertices (i.e., $0$-cells) correspond via $\psi$ to the $\sigma$-compatible discrete orientations of $M$. 
\end{enumerate}
\end{proposition}

\begin{proof} The only non-trivial part of (1) is the first half. By Proposition \ref{prop:ctslatticepointprop} and \ref{prop:continuouscompatible}, every point of $Z_A$ is of the form $\psi(\mathcal{O})$ for some $\sigma$-compatible continuous orientation $\mathcal{O}$. Hence by Proposition \ref{prop:uniqueext}, it suffices to show that if the set $\hat{B}$ of bi-oriented elements in $\mathcal{O}$ do not form a basis, then we can bi-orient some element in $\mathcal{O}$ while maintaining $\sigma$-compatiblility; by induction, we will end up with a bi-oriented basis $B$, which implies that $\psi(\mathcal{O})$ is a limit point of $Z^\circ(B)$.

Suppose that for every $e\not\in\hat{B}$ such that $\hat{B}\cup\{e\}$ is independent in $M$, bi-orienting $e$ in $\mathcal{O}$ will cause the new continuous orientation $\mathcal{O}_e$ to no longer be $\sigma$-compatible. Then every such $\mathcal{O}_e$ is compatible with $-\sigma(C_e)$ for some circuit $C_e$ containing $e$. Pick, among all such elements $e$ and circuits $C_e$, the pair that maximizes $w\cdot{\bf v}_{\sigma(C_e)}$, where we always choose the normalized ${\bf v}_{\sigma(C_e)}$ whose $e$-th coordinate is $\sigma(C_e)(e)$. The circuit $C_e$ must contain another element $f\not\in\hat{B}$ such that $\hat{B}\cup\{f\}$ is independent in $M$, so there exists some circuit $C_f$ containing $f$ such that $\mathcal{O}_f$ is compatible with $-\sigma(C_f)$. The signs of $\sigma(C_e)$ and $\sigma(C_f)$ over $f$ are different, so we can choose a suitable positive multiple ${\bf v}_f$ of ${\bf v}_{\sigma(C_f)}$ such that the $f$-th coordinates of ${\bf v}_f$ and ${\bf v}_{-\sigma(C_e)}$ are equal.

By Lemma~\ref{lem:circuitdecomposition}, ${\bf v}_{-\sigma(C_e)}-{\bf v}_f$ can be written as a sum $\sum_{i=1}^k {\bf v}_{C_i}$ of signed circuits. Each such signed circuit $C_i$ that does not contain $e$ must be compatible with $\mathcal{O}$ (hence $\sigma$), while those signed circuits that contain $e$ would at least be compatible with $\mathcal{O}_e$. Since $w\cdot(\sum_{i=1}^k {\bf v}_{C_i})=w\cdot({\bf v}_{-\sigma(C_e)}-{\bf v}_f)<0$, some $C_i$ is not compatible with $\sigma$ (hence $\mathcal{O}$), thus they contain $e$. In particular, the sign of the $e$-th coordinate of ${\bf v}_{-\sigma(C_e)}-{\bf v}_f$ agrees with $-\sigma(C_e)$. But as the signs of $\sigma(C_e)(e)$ and $\sigma(C_f)(e)$ are different, the absolute value of the $e$-th coordinate of ${\bf v}_f$ is at most the absolute value of the $e$-th coordinate of ${\bf v}_{-\sigma(C_e)}$, which is 1.

Without loss of generality, the circuits containing $e$ are $C_1,C_2,\ldots,C_j$. We choose ${\bf v}_{C_i}$'s so that their $e$-th coordinates equal $-\sigma(C_e)(e)$, and rewrite ${\bf v}_{-\sigma(C_e)}-{\bf v}_f$ as $\sum_{i=1}^k \lambda_i {\bf v}_{C_i}$ for some $\lambda_i>0$. By comparing $e$-th coordinate, $\sum_{i=1}^j\lambda_i\leq 1$. Now we have 
$$
w\cdot(\sum_{i=1}^j \lambda_i {\bf v}_{C_i})=w\cdot \left({\bf v}_{-\sigma(C_e)}-{\bf v}_f-\sum_{i=j+1}^k \lambda_i {\bf v}_{C_i} \right)<-w\cdot{\bf v}_{\sigma(C_e)}<0,
$$
i.e., there exists some $C_i$ with $i\leq j$ that is compatible with $\mathcal{O}_e$, disagrees with $\sigma$, and $w\cdot{\bf v}_{\sigma(C_i)}>w\cdot{\bf v}_{\sigma(C_e)}$, contradicting our choice of $C_e$.

\medskip

For (2), $Z^\circ(B)$ can be identified, up to an affine linear transformation of full rank, with the open paralleletope $(0,1)^B$, where the $e$-th coordinate of $\psi({\mathcal O})\in Z^\circ(B)$ is the value $\widehat{{\mathcal O}(e)}$. Thus $Z(B)$ can be identified with the paralleletope $[0,1]^B$ in an analogous manner, and restricting to a face of $Z(B)$ of codimension $i$ can be described as orienting $i$ elements in $B$. This gives a combinatorial description of each face as the combinatorial type of any $\sigma$-compatible orientation in its relative interior, and conversely, every combinatorial type of $\sigma$-compatible orientation determines a unique face. Hence if the relative interiors of two faces intersect, then the two faces must be equal, showing that the collection of $Z(B)$'s gives a polyhedral subdivision of $Z_A$.
\end{proof}

\begin{remark}
We give a description of the incidence relation between cells in the polyhedral subdivision; we will not give a proof as we will not make use of it.
Let $B$ be a basis, and let $F_e$ be a facet of $Z(B)$ corresponding to orienting some $e\in B$.
Let $\mathcal{O}$ be a continuous orientation obtained from orienting $e$ in any continuous orientation of the form $\mathcal{O}(B,b)$.
Then either (1) $C^*(B,e)$ is a positive cocircuit in $\mathcal{O}$, in which case $F_e$ lies on the boundary of $Z_A$, or 
(2) there exists a unique element $f\in C^*(B,e)\setminus e$ such that the orientation obtained by reversing $f$ in $\mathcal{O}$ is also $\sigma$-compatible, in which case $F_e$ is a facet of $Z((B\setminus\{e\})\cup\{f\})$.
\end{remark}

\begin{figure}[ht!]
\begin{center}
    \includegraphics[width=15cm, height = 15cm, keepaspectratio]{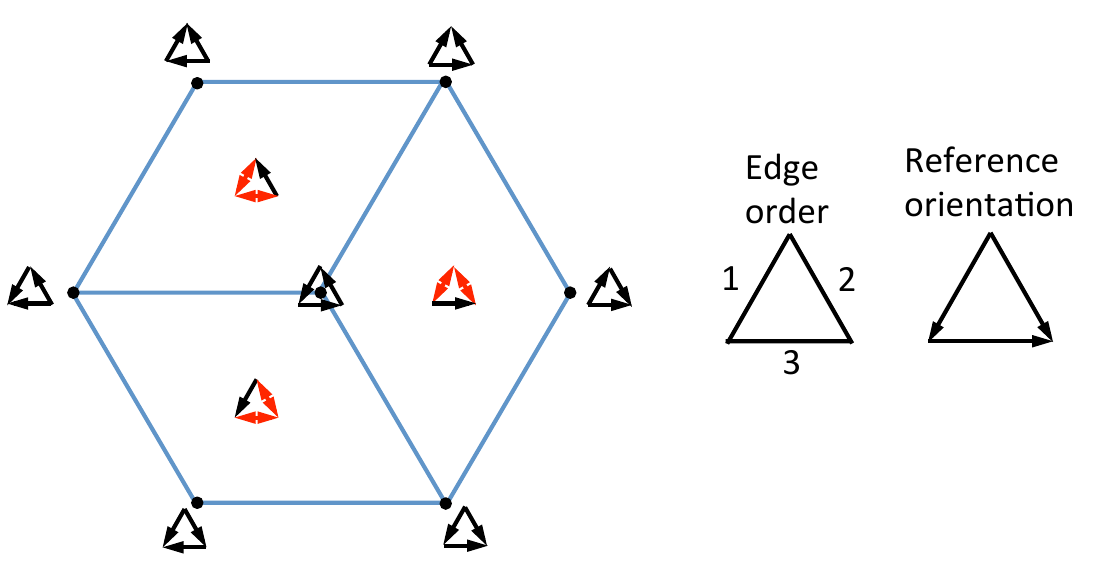}
\end{center}
  \caption{The subdivision of the zonotope associated to $K_3$ as described in Proposition \ref{prop:zonotopedecomp} using $\sigma$ induced by the total order and reference orientation on the right as described in Example \ref{edgeorder}.
  The red edges are bi-oriented.}
  \label{zonotopeK3_6}
\end{figure}

\subsection{Geometric interpretation of the combinatorial map}

Let $\sigma,\sigma^*$ be acyclic signatures of $C(M)$ and $C^*(M)$, respectively.  By Lemma~\ref{lem:gordan}, there exists $w \in {\mathbb R}^E$ that induces both $\sigma$ and $\sigma^*$.
Our next goal is to show that the combinatorially defined basis-to-orientation map $\hat{\beta}$ (whose definition depends on $\sigma$ and $\sigma^*$) can be interpreted geometrically as a ``shifting map''.

\medskip

To present the calculation in our proof more clearly, for the rest of \S\ref{sec:proof}, we will work in the {\em cocircuit space} $V^*(M)$ of $M$, which is the ${\mathbb R}$-span of $C^*(M)$ (and is equal to the row space of $A$). Let $\pi_{V^*(M)}$ be the orthogonal projection from $\RR^E$ onto $V^*(M)$ and let $\{u_e:e\in E\}$ be the standard basis for $\RR^E$. Consider the (row) zonotope $\widetilde{Z_A} := \{ \sum_{e\in E} c_e \pi_{V^*}(u_e) \; : \; 0 \leq c_e \leq 1 \}\subset V^*(M)$.
The following lemma shows that $\widetilde{Z_A}$ and the previously defined zonotope $Z_A$ are essentially equal:

\begin{lemma} \label{lem:RC_zonotope_isomorphic}
The map $L:v\mapsto Av$ is a lattice points preserving isomorphism between $V^*(M)$ and $\RR^r$ taking $\widetilde{Z_A}$ to $Z_A$.
\end{lemma}

\begin{proof}
Since $AA^T$ has full rank, $V^*(M)=\{A^Tz:z\in\RR^r\}$ is isomorphic to $\RR^r$ via $L$. By simple linear algebra, we have $L(\pi_{V^*}(u_e))=L(A^T(AA^T)^{-1}Au_e)=A_e$.  
Thus $L(\sum_{e\in E} c_e \pi_{V^*}(u_e))=\sum_{e\in E} c_e A_e$ and $L(\widetilde{Z_A})=Z_A$. $L$ preserves lattice points because $A$ is totally unimodular.
\end{proof}

In particular, the subdivision of $Z_A$ constructed in \S\ref{sec:polyhedralsubdivision} 
induces a corresponding subdivision of $\widetilde{Z_A}$.  We denote by $\widetilde{Z(B)}$ the cell $L^{-1}(Z(B))$ in $\widetilde{Z_A}$.

\medskip

The key to defining the shifting map is the following lemma:

\begin{lemma} \label{lem:phiwelldefined}
If $w'$ is the orthogonal projection of $w$ onto $V^*(M)$, then for all sufficiently small $\epsilon > 0$ the image of $\widetilde{Z(B)}$ under the map $v \mapsto v + \epsilon w'$ contains a unique point corresponding (via $\psi$) to a $\sigma$-compatible discrete orientation ${\mathcal O}_B$.
\end{lemma}

\begin{proof} By Proposition \ref{prop:zonotopedecomp}, the vertices of each $\widetilde{Z(B)}$ correspond to $\sigma$-compatible discrete orientations. It therefore suffices to prove that $w'$ does not lie in the affine span of any facet of $\widetilde{Z(B)}$.
The affine span of a facet of $\widetilde{Z(B)}$ is spanned by directions of the form $\pi_{V^*}(e)$ for $e\in\hat{B}$ where $\hat{B}\subsetneq B$.
Since $|\hat{B}|<r$, there is a cocircuit $K$ of $M$ avoiding $\hat{B}$. Any direction $v := \sum_{e\in\hat{B}}\lambda_e\pi_{V^*}(e)$ in the span satisfies $\langle v, {\bf v}_{\sigma^*(K)}\rangle
=\sum_{e\in\hat{B}}\lambda_e\langle e,{\bf v}_{\sigma^*(K)}\rangle=0$. On the other hand, since $w$ induces $\sigma^*$, $\langle w', {\bf v}_{\sigma^*(K)}\rangle=\langle w,{\bf v}_{\sigma^*(K)}\rangle>0$.
\end{proof}

We define $\phi$ to be the map that takes a basis $B$ to the orientation ${\mathcal O}_B$ defined in Lemma \ref{lem:phiwelldefined}.

\begin{theorem} \label{thm:phibetaequal}
The map $\phi$ coincides with the previously defined map $\hat{\beta}$.
\end{theorem}

\begin{proof}

Let $B$ be a basis. Then $\phi(B)$ can be obtained by orienting each (bi-oriented) basis element from a continuous $\sigma$-compatible orientation in the interior of $\widetilde{Z(B)}$ (which is of the form $\mathcal{O}(B,b)$), so by the greedy procedure described in Proposition \ref{prop:uniqueext}, the elements outside $B$ are oriented according to their fundamental circuits, hence $\phi(B)$ agrees with $\hat{\beta}(B)$ outside $B$.

For elements inside $B$, we work with the basis $\{\pi_{V^*}(u_e):e\in B\}$ for $V^*(M)$ and write $w'=\sum_{e\in B}w_e\pi_{V^*}(u_e)$. 
Identifying $\widetilde{Z(B)}$ with $[0,1]^B$ and the vertices of $\widetilde{Z(B)}$ with $\{0,1\}^B$. If a vertex $v$ is identified with $(s_e:e\in B)$, then it corresponds to a $\sigma$-compatible discrete orientation where each element $e\in B$ is oriented in agreement with (resp. opposite to) its reference orientation when $s_e=1$ (resp. $s_e=0$). The cell $\widetilde{Z(B)}$ will contain $v$ after shifting if and only if the sign pattern of the $s_e$'s agrees with the sign pattern of the $w_e$'s, i.e., if and only if $s_e=1$ precisely when $w_e>0$.

Let $f \in B$, and let $K$ be the fundamental cocircuit of $f$ with respect to $B$. By a calculation similar to the above, 
\[
0<\langle w',{\bf v}_{\sigma^*(K)}\rangle=\sum_{e\in B}w_e\langle u_e, {\bf v}_{\sigma^*(K)}\rangle=w_f\langle u_f, {\bf v}_{\sigma^*(K)}\rangle,
\]
as $f$ is the unique element in $B\cap K$. If $w_f>0$, then $\langle u_f, {\bf v}_{\sigma^*(K)}\rangle>0$ and the reference orientation of $f$ agrees with $\sigma^*(K)$, i.e., the orientation of $f$ in $\phi(B)$ is the same as the reference orientation of $f$. From the last paragraph, $f$ is oriented according to its reference orientation in $\hat{\beta}(B)$ as well, because $w_f>0$. A similar analysis in the case where $w_f<0$ implies also that $\phi(B)(f)=\hat{\beta}(B)(f)$.
\end{proof}

\begin{proposition}\label{prop:betaCCM}
Let $B$ be a basis. Then $\hat{\beta}(B)$ is $(\sigma,\sigma^*)$-compatible.
\end{proposition}

\begin{proof} Since $\phi(B)$ is $\sigma$-compatible, $\hat{\beta}(B)$ is also $\sigma$-compatible by Theorem \ref{thm:phibetaequal}.  And since the procedure described in Theorem \ref{thm:mainbijectionforregularmatroids} is symmetric with respect to circuits and cocircuits, a dual argument shows that $\hat{\beta}(B)$ is $\sigma^*$-compatible.
\end{proof}

\begin{theorem}[Theorem~\ref{thm:realizablemainthm}] \label{thm:betabijective}
The map $\hat{\beta} : {\mathcal B}(M) \to {\mathcal X}(M)$ is a bijection.
\end{theorem}

\begin{proof} $\hat{\beta}=\phi$ is injective for the simple geometric reason that a vertex can only be contained in the interior of at most one cell $\widetilde{Z(B)}$ after shifting. To prove the surjectivity part, we need to show that for every $(\sigma,\sigma^*)$-compatible orientation $\mathcal{O}$, there exists a continuous orientation $\mathcal{O}'$ such that the displacement from $\mathcal{O}'$ to $\mathcal{O}$, interpreted as points of $\widetilde{Z_A}$, is $\pi_{V^*}(w)$ (here we assume $w$ is sufficiently short). For simplicity, we negate suitable columns of $A$ in order to assume without loss of generality that $\mathcal{O}\equiv {\bf 1}$, and we modify $w$ accordingly. For such to be determined $\mathcal{O}'$, denote by ${\bf f}_e\geq 0$ the difference between $\widehat{\mathcal{O}(e)}=1$ and $\widehat{\mathcal{O}'(e)}$. By simple linear algebra, our condition on $\mathcal{O}'$ in terms of displacement becomes $A{\bf f}=Aw$, hence $\mathcal{O}'$ exists if and only if the linear program
\begin{equation}\label{inverse_LP}
\min\{{\bf 1}^T{\bf f}:A{\bf f}=Aw, {\bf f}\geq {\bf 0}\}
\end{equation}
is feasible. But the $\sigma^*$-compatible condition implies ``if $z^TA\geq 0$, then $(z^TA)w\geq 0$'', which is the same as ``there exists no $z$ such that $z^TA\geq 0, z^T(Aw)<0$'', by the Farkas lemma, the latter condition is equivalent to the existence of some ${\bf f}\geq 0$ such that $A{\bf f}=Aw$.
\end{proof}

\begin{figure}[ht!]
\begin{center}
    \includegraphics[width=1\textwidth]{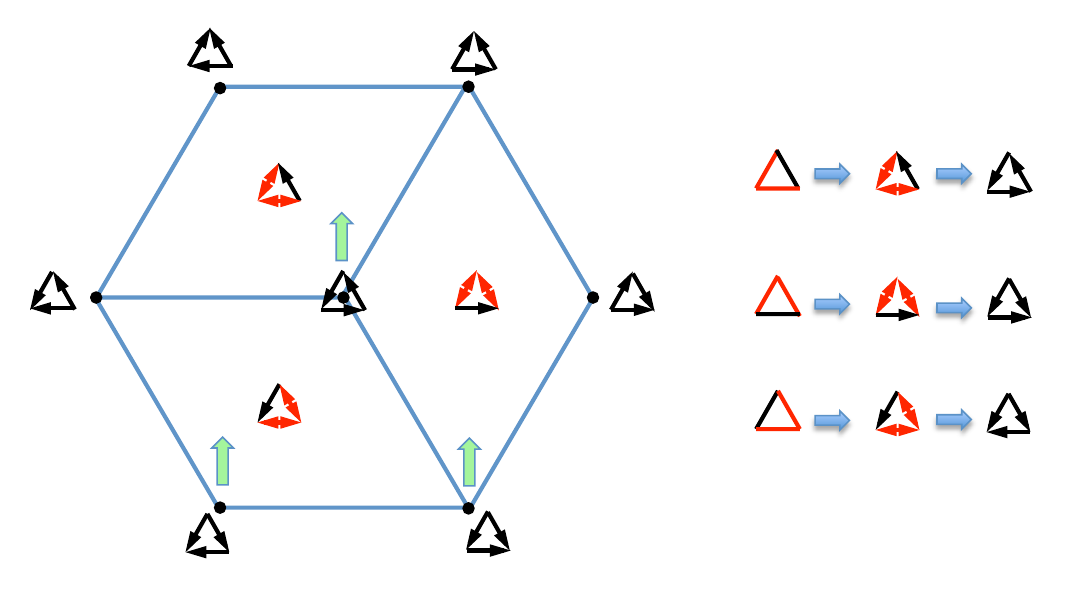}
\end{center}
  \caption{An example of the bijection for $K_3$ using the pair $(\sigma, \sigma^*)$ induced by the total order and reference orientation from Figure \ref{zonotopeK3_6}.}
\end{figure}

\subsection{Computability of the inverse map} \label{sec:bijectioncomputable}

We now describe an inverse algorithm which furnishes an inverse to the map $\phi$, and hence to $\hat{\beta}$.
Again we assume the inputted $(\sigma,\sigma^*)$-compatible discrete orientation $\mathcal{O}$ is equal to ${\bf 1}$ for simplicity.
Suppose $\mathcal{O}$ was shifted into the cell $\widetilde{Z(B)}$ after moving by a displacement of $-\pi_{V^*}(w)$. By solving the linear program (\ref{inverse_LP}) we obtain a continuous orientation $\mathcal{O}'$ (resp. ${\bf f}$) in the cell $Z(B)$.
Therefore it remains to find the $\sigma$-compatible continuous orientation $\mathcal{O}''$ equivalent to $\mathcal{O}'$, and the desired basis $B$ will then be the set of bi-oriented elements in $\mathcal{O}''$.

To do so, we solve the linear program 
\begin{equation}\label{maxflow_LP}
\max\{w^T{\bf y}: A{\bf y}=0, {\bf f}_e-1\leq {\bf y}_e\leq {\bf f}_e,\forall e\}.
\end{equation}
Let $\tilde{\bf y}$ be an optimal solution.
Consider the continuous orientation $\mathcal{O}'':=\mathcal{O}'+2\tilde{\bf y}$, we claim this is the continuous orientation we are looking for. 
The conditions in the linear program guarantee that $\mathcal{O}''$ is a valid continuous orientation circuit reversal equivalent to $\mathcal{O}''$, and it is $\sigma$-compatible: indeed, if $\mathcal{O}''$ is compatible with some $-\sigma(C)$, then one can easily check that $\tilde{\bf y}+\delta {\bf v}_{\sigma(C)}$ is also a feasible solution for sufficiently small $\delta>0$, contradicting the optimality of $\tilde{\bf y}$.

Since linear programming admits a polynomial-time algorithm \cite{schrijver1986LP}, the linear program (\ref{maxflow_LP}), together with the dual version of it, imply the following:

\begin{proposition} \label{prop:CCMalgo}
There is a polynomial-time algorithm to compute the unique $(\sigma,\sigma^*)$-compatible continuous orientation circuit-cocircuit equivalent to a given continuous orientation.
\end{proposition}

Summarizing the discussion, we have the following theorem:

\begin{theorem} \label{prop:invalgo}
There is a polynomial-time algorithm to compute the inverse of $\hat{\beta}$.
\end{theorem}

\section{The discrete circuit-cocircuit reversal system for a regular matroid and its Jacobian}\label{discretestuffforregularmatroids}

We now return to the setting of regular matroids.  
Throughout this section, $M$ will denote a regular matroid on $E$ and $A$ will be a totally unimodular matrix representing $M$.
We will investigate the (original) discrete version of circuit(-cocircuit) reversal system which was introduced by Gioan \cite{gioan2007enumerating,gioan2008circuit}, and show that the $\sigma$-compatible discrete orientations also give distinguished representatives for this system. Moreover, we will show that discrete circuit-reversal classes correspond to lattice points of the zonotope $Z_A$ (which by Proposition \ref{prop:zonotopedecomp} are precisely the vertices of the zonotopal subdivision $\Sigma$).  Finally, we show that the discrete circuit-cocircuit reversal system is canonically a torsor for $\Jac(M)$.

\subsection{The discrete circuit-cocircuit reversal system}\label{dccrs}

For totally unimodular matrices, we have the following integral version of Lemma~\ref{lem:circuitdecomposition}:

\begin{lemma} \label{lem:regularcircuitdecomposition}
Let $u\in\Lambda_A(M)$. Then $u$ can be written as an integral sum of signed circuits (as elements of $\Lambda_A(M)$) $\sum \lambda_C C$ with $\lambda_C>0$, such that each $C$ is conformal to $u$. In particular, if $u$ is a $\{ 0, \pm 1 \}$-vector, then the $\lambda_C$'s are 1 and the $C$'s are disjoint.
\end{lemma}

\begin{proof} Without loss of generality, we may assume $u\geq 0$. We first pick a signed circuit $C$ conformal to $u$ as in the statement of Lemma~\ref{lem:circuitdecomposition}. By total unimodularity, ${\bf v}_C$ can be chosen as a $\{0,1\}$-vector, and we choose $\lambda_C$ to be the maximum number such that $u-\lambda_C C\geq 0$. In such case $\lambda_C$ must be an integer and the support of $u-\lambda_C C\in\Lambda_A(M)$ is strictly contained in the support of $C$. Proceed by induction to obtain the desired decomposition. The second assertion follows easily from the first.
\end{proof}

A {\em discrete orientation} ${\mathcal O}$ of $M$ is a function $E \to \{-1,1\}$.
A discrete orientation ${\mathcal O}$ is {\em compatible} with a signed circuit $C$ of $M$ if ${\mathcal O}(e) \neq -C(e)$ for all $e$ in the support of $C$.

If $\mathcal{O}$ is a discrete orientation and $C$ is a signed circuit compatible with $\mathcal{O}$, we can perform a (discrete) {\em circuit reversal} taking $\mathcal{O}$ to the orientation $\mathcal{O}'$ defined by $\mathcal{O}'(e) = -\mathcal{O}(e)$ if $e$ is in the support of $C$ and $\mathcal{O}'(e) = \mathcal{O}(e)$ otherwise.
The {\em discrete circuit reversal system} is the equivalence relation on the set $CO(M)$ of all discrete orientations of $M$ generated by all possible discrete circuit reversals.
We can make the same definitions for cocircuits by replacing $M$ with its dual.

\medskip

We first state a basic fact about orientations (which is true more general for oriented matroids)\cite[Corollary 3.4.6]{bjorner1999oriented}:

\begin{proposition} \label{prop:orientdecomp}
Given an orientation $\mathcal{O}$ of $M$ and $e \in E$, exactly one of the following holds:
\begin{enumerate}
\item There is a signed circuit $C$ of $M$ with $e \in {\rm supp}(C)$ such that ${\mathcal O}(f) = C(f)$ for every $f$ in the support of $C$.  In this case we say that $e$ belongs to the {\em circuit part} of $\mathcal{O}$.
\item There is a signed cocircuit $C^*$ of $M$ with $e \in {\rm supp}(C^*)$ such that ${\mathcal O}(f) = C^*(f)$ for every $f$ in the support of $C^*$. In this case we say that $e$ belongs to the {\em cocircuit part} of $\mathcal{O}$.
\end{enumerate}
\end{proposition}

\begin{proposition} \label{prop:latticepointprop}
The map $\psi$ from \S\ref{sec:continuousreversals} induces a bijection between discrete orientations of $M$ modulo discrete circuit reversals and lattice points of $Z_A$.
\end{proposition}

\begin{proof}
As the columns of $A$ are integral, $\psi$ takes an orientation of $M$ to a lattice point of $Z_A$; conversely, for any lattice point $y\in Z_A$, $A\hat{\alpha}=y, 0\leq \hat{\alpha}_i\leq 1\ \forall i$ has a solution $\hat{\alpha}$, which can be chosen to be integral by the total unimodularity of $A$, hence it corresponds to an orientation.  Thus the image of $\psi$ is precisely the set of lattice points of $Z_A$. By the orthogonality of circuits and cocircuits, two orientations in the same circuit-reversal class map to the same point of $Z_A$.  Conversely, suppose $\psi(\mathcal{O})=\psi(\mathcal{O}')$. 
By Lemma~\ref{lem:regularcircuitdecomposition},
$\mathcal{O}-\mathcal{O}'$ can be written as a sum of disjoint signed circuits in which each signed circuit is compatible with $\mathcal{O}$, and $\mathcal{O}$ can be transformed to $\mathcal{O}'$ via the corresponding circuit-reversals in any order.
\end{proof}

\begin{proposition} \label{prop:discretecompatible}
Each discrete circuit-reversal class of discrete orientations of $M$ contains a unique $\sigma$-compatible discrete orientation.
\end{proposition}

\begin{proof} The uniqueness assertion follows from Lemma~\ref{lem:regularcircuitdecomposition} and a similar argument as in Proposition~\ref{prop:continuouscompatible}. For existence, start with any orientation $\mathcal{O}$ in the class and reverse some signed circuit $C$ compatible with $\mathcal{O}$ but not compatible with $\sigma$. We claim that the process will eventually stop. Indeed, suppose not: since the number of discrete orientations of $M$ is finite, the orientation will without loss of generality return to $\mathcal{O}$ after reversing some signed circuits $C_1,\ldots,C_k$ in that order (the circuits might not be distinct). Then $-C_1-\cdots-C_k=0$, which means that $\sigma(C_1)+\cdots+\sigma(C_k)=0$, contradicting the acyclicity of $\sigma$. 
\end{proof}

\begin{corollary} \label{coro:latticepoint_are_vertex}
The lattice points of $Z_A$ are exactly the vertices of the subdivision $\Sigma$.
\end{corollary}

\begin{proof}
This follows from Propositions \ref{prop:latticepointprop}, \ref{prop:discretecompatible} and \ref{prop:zonotopedecomp}.  
\end{proof}

\begin{figure}[ht!]
\begin{center}
   \includegraphics[width=10cm, height = 10cm, keepaspectratio]{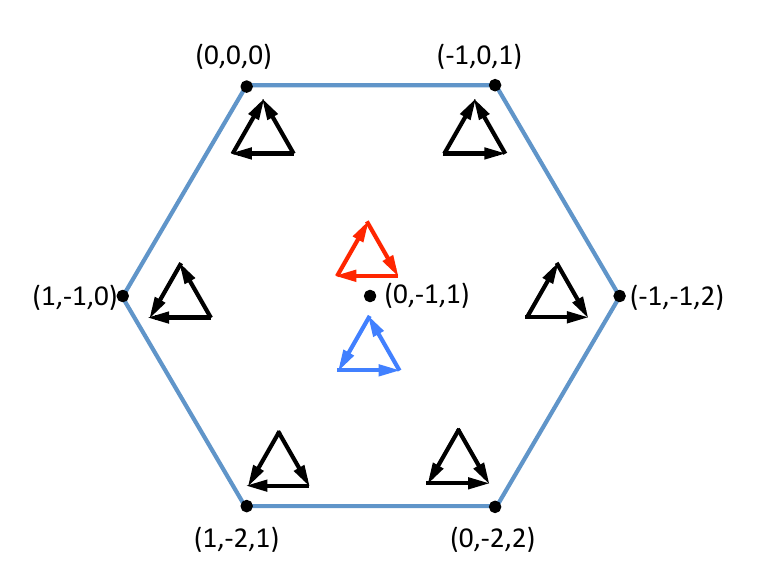}
\end{center}
  \caption{The zonotope associated to $K_3$ and the circuit reversal classes associated to its lattice points by the map $\psi$ from Proposition~\ref{prop:ctslatticepointprop}.  Taking the acyclic signature $\sigma$ from Figure \ref{zonotopeK3_6}, the cycle in blue is $\sigma$-compatible, while the cycle in red is not.  (Note that we are using the full adjacency matrix of $K_3$ to define the zonotope, cf.~Remark~\ref{rmk:graphicmatrix}.) 
  }
  \label{zonotopeK3_3}
\end{figure}

Let $\chi:{\mathcal X}(M) \rightarrow {\mathcal G}(M)$ be the map which associates to each $(\sigma,\sigma^*)$-compatible orientation the discrete circuit-cocircuit reversal class which it represents.

\begin{theorem} [Part (2) of Theorem \ref{thm:SScompatibleorientation}] \label{thm:chibijective}
The map $\chi$ is a bijection.
\end{theorem}

\begin{proof}
This follows directly from Propositions \ref{prop:discretecompatible} and \ref{prop:orientdecomp}
\end{proof}

\begin{corollary} [Theorem~\ref{thm:mainbijectionforregularmatroids}]
The map $\beta:B(M)\rightarrow\mathcal{G}(M)$ given by $B\mapsto [\mathcal{O}(B)]$ is a bijection.
\end{corollary}

\begin{proof}
The map $\hat{\beta}:B\mapsto \mathcal{O}(B)$ is a bijection between $B(M)$ and $\mathcal{X}(M;\sigma,\sigma^*)$ by Theorem~\ref{thm:betabijective}. Now compose this map with $\chi$.
\end{proof}

\subsection{The circuit-cocircuit reversal system as a $\Jac(M)$-torsor}
\label{sec:Torsor}

In this section we will define a natural action of $\Jac(M)$ on the set ${\mathcal G}(M)$ of circuit-cocircuit equivalence classes of orientations of $M$ and prove that the action is simply transitive. We will also discuss an efficient algorithm for computing this action, along with an application to randomly sampling bases of $M$.

\subsection{Definition of the action}

Recall from (\ref{eq:canonisom}) that $\Jac(M)$ can be identified with $\frac{\mathbb{Z}^E}{\Lambda_A(M)\oplus \Lambda_A^*(M)}$. Note that such group is generated by $[\overrightarrow{e}],e\in E$ (here we use an overhead arrow to emphasize that we are keeping track of orientations). 

The group action $\Jac(M)\circlearrowright\mathcal{G}(M)$ is defined by linearly extending the following action of each generator $[\overrightarrow{e}]$ on circuit-cocircuit reversal classes: pick an orientation $\mathcal{O}$ from the class so that $e$ is oriented as $\overrightarrow{e}$ in $\mathcal{O}$, reverse the orientation of $e$ in $\mathcal{O}$ to obtain $\mathcal{O}'$, and set $[\overrightarrow{e}]\cdot[\mathcal{O}]=[\mathcal{O}']$.
This action generalizes the one defined in terms of {\em path reversals} by the first author in the graphical case \cite[Section 5]{backman2014riemann}.

\begin{figure}[ht!]
\begin{center}
    \includegraphics[width=10cm, height = 10cm, keepaspectratio]{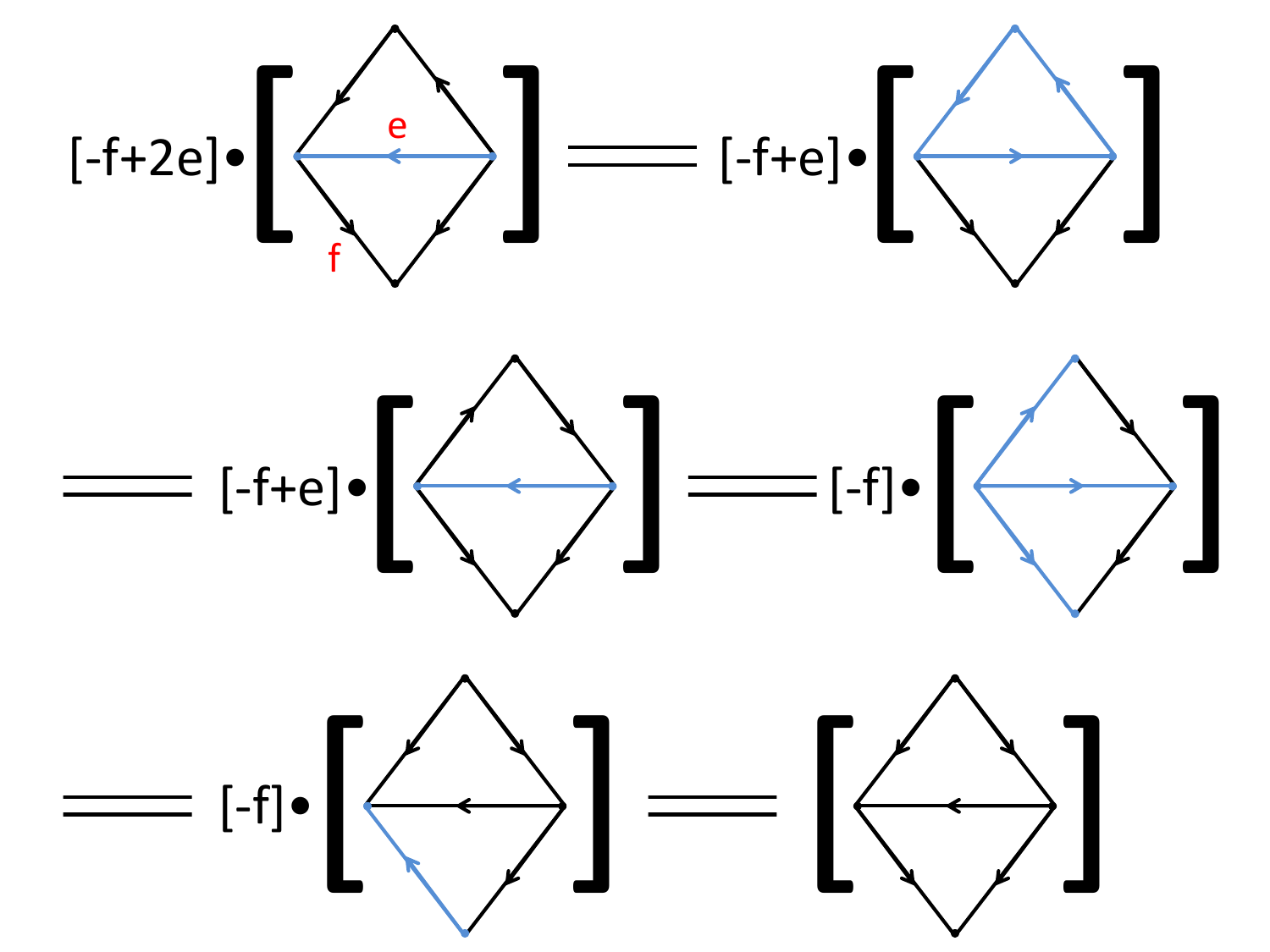}
\end{center}
  \caption{Example of the torsor. Here the reference orientations of $e,f$ are the same as the orientation we begin with.}
  \label{Torsor_Example}
\end{figure}

Our main goal for the rest of this section will be to prove:

\begin{theorem} \label{Thm:Action}
The group action $\circlearrowright$ is well-defined and simply transitive. 
\end{theorem}

\begin{remark}
For ease of exposition, in the rest of this section we will use the term {\em positive} (co)circuit (with respect to an orientation $\mathcal{O}$) to denote a signed (co)circuit that is compatible with $\mathcal{O}$. Furthermore, given an orientation $\mathcal{O}$ and a subset $X\subset E$, we denote by $_{-X}\mathcal{O}$ the orientation obtained by reversing elements of $X$ in $\mathcal{O}$.  For a (co)circuit $C$ of $\mathcal{O}$, we say that $_{-X}C$ is positive if $C$ is a positive (co)circuit of $_{-X}\mathcal{O}$. Finally, we denote by $\rchi_X$ the $\{0,1\}$-characteristic vector whose support is $X$.
\end{remark}

\subsection{The action is well-defined}

In order to show that the action of $\Jac(M)$ on $\mathcal{G}(M)$ is well-defined, we first show that the corresponding action (which by abuse of notation we continue to write as $\circlearrowright$) of $\ZZ^E$ on $\mathcal{G}(M)$ is well-defined, then that the action descends to the quotient by $\Lambda_A(M)\oplus \Lambda_A^*(M)$.

\begin{lemma} \label{Lem:ExtendCC}
Let $e\in E$, and suppose $X\subset E\setminus e$ is a positive cocircuit in $\mathcal{O}\setminus e$ but not in $\mathcal{O}$. Then $Y:=X\cup\{e\}$ is a cocircuit in $\mathcal{O}$, and either $Y$ or $_{-e}Y$ is positive.
\end{lemma}

\begin{proof}
By assumption, $w^TA|_{E\setminus e}=\rchi_X$ for some $w$. Hence $w^TA=\rchi_X+\lambda\rchi_{\{e\}}$ for some $\lambda\neq 0$. By the dual of Proposition \ref{lem:circuitdecomposition}, $Y$ contains a cocircuit $D$. 
If $D\cap X=\emptyset$, then $D=\{e\}$, which in turn shows that $X$ itself is a positive cocircuit. Now we must have $X\subset D$, or otherwise $D\cap X\subsetneq X$ would be a cocircuit in $M\setminus e$. Therefore $Y=D$ is a cocircuit, and $\lambda=\pm 1$, i.e. either $Y$ or $_{-e}Y$ is positive. 
\end{proof}

\begin{lemma} \label{Lem:AlmostPosCC}
Suppose $e \in M$ is contained in some positive circuit of $\mathcal{O}$, and that $Y$ is a subset of $E$ containing $e$ such that $_{-e}Y$ is a positive cocircuit. Then any positive circuit containing $e$ intersects $Y$ in exactly two elements.
\end{lemma}

\begin{proof}
Let $C$ be a positive circuit containing $e$. By assumption, there exists a vector $v$ such that $v^TA=\rchi_{Y-e}-\rchi_{\{e\}}$.  Then  $0=v^TA\rchi_C=|(Y-e)\cap C|-1$, i.e., $Y$ intersects $C$ in $e$ together with exactly one more element.
\end{proof}

\begin{proposition} \label{Prop:Torsor_WellDef}
For every $[\mathcal{O}]\in\mathcal{G}(M)$ and oriented element $\overrightarrow{e}$, there exists $\tilde{\mathcal{O}}\in [\mathcal{O}]$ so that $e$ is oriented as $\overrightarrow{e}$ in $\tilde{\mathcal{O}}$.
\end{proposition}

\begin{proof}
By Proposition~\ref{prop:orientdecomp}, $e$ is either contained in a positive circuit or cocircuit $C$. If $e$ is not already oriented as $\overrightarrow{e}$ in $\mathcal{O}$, reverse $C$.
\end{proof}

\begin{proposition} \label{Prop:IndepRep}
The action of $\overrightarrow{e}$ on $[\mathcal{O}]$ is independent of which orientation we choose.
\end{proposition}

\begin{proof}
Suppose $\mathcal{O}\sim\mathcal{O}'$ and they agree on $e$, then $\mathcal{O}$ and $\mathcal{O}'$ differ by a disjoint union of positive circuits and cocircuits which do not contain $e$ by Lemma~ \ref{lem:circuitdecomposition} and its dual. Thus $_{-e}\mathcal{O}\sim_{-e}\mathcal{O}'$ using the same reversals.
\end{proof}

\begin{proposition} \label{Prop:commute_action}
For any $\overrightarrow{e},\overrightarrow{f}\in \mathbb{Z}^E$ and $[\mathcal{O}]\in\mathcal{G}(M)$, $\overrightarrow{e}\cdot(\overrightarrow{f}\cdot [\mathcal{O}])=\overrightarrow{f}\cdot(\overrightarrow{e}\cdot[\mathcal{O}])$. Hence it is valid to extend $\cdot$ linearly, and $\circlearrowright$ is indeed a group action of $\ZZ^E$ on $\mathcal{G}(M)$ .
\end{proposition}

\begin{proof}
The statement is tautological if $\overrightarrow{e}=\overrightarrow{f}$. If $\overrightarrow{e}=-\overrightarrow{f}$, then without loss of generality the orientation of $e$ in $\mathcal{O}$ is $\overrightarrow{e}$.  Let $C$ be a positive (co)circuit containing $e$. Then $\overrightarrow{f}\cdot(\overrightarrow{e}\cdot[\mathcal{O}])
=[\mathcal{O}]=[_{-C}\mathcal{O}]=\overrightarrow{e}\cdot(\overrightarrow{f}\cdot [\mathcal{O}])$.

Otherwise $e\neq f$. We may again assume that $e$ is oriented as $\overrightarrow{e}$ in $\mathcal{O}$. The statement is easy if there exists some positive (co)circuit in $\mathcal{O}$ that contains $f$ but not $e$, as we can reverse it and obtain an orientation in which the orientations of $e,f$ are already $\overrightarrow{e},\overrightarrow{f}$.
So without loss of generality $e,f$ are in the circuit part of $\mathcal{O}$ and every positive circuit containing $f$ also contains $e$; fix any such positive circuit $C$. $f$ must be in some positive cocircuit $D'$ of $\mathcal{O}-e$, since otherwise $f$ is in some positive circuit of $\mathcal{O}-e$, which is a positive circuit in $\mathcal{O}$ avoiding $e$. By Lemma~\ref{Lem:ExtendCC},  $D:=D'\cup\{e\}$ is a cocircuit in $\mathcal{O}$ and $_{-e}D$ is positive, and by Lemma \ref{Lem:AlmostPosCC}, we know that $C\cap D=\{e,f\}$.

On one hand we have $\overrightarrow{f}\cdot(\overrightarrow{e}\cdot[\mathcal{O}])
=\overrightarrow{f}\cdot[_{-e}\mathcal{O}]
=\overrightarrow{f}\cdot[_{-(D-e)}\mathcal{O}]
=[_{-(D-\{e,f\})}\mathcal{O}]$. On the other hand, 
$\overrightarrow{e}\cdot(\overrightarrow{f}\cdot[\mathcal{O}])
=\overrightarrow{e}\cdot(\overrightarrow{f}\cdot[_{-C}\mathcal{O}])
=\overrightarrow{e}\cdot[_{-(C-f)}\mathcal{O}]
=\overrightarrow{e}\cdot[_{-(C\cup D-e)}\mathcal{O}]=[_{-(C\cup D)}\mathcal{O}]$. But $C$ is positive in $_{-(C\cup D)}\mathcal{O}$, so $[_{-(C\cup D)}\mathcal{O}]=[_{-(C\cup D)\triangle C}\mathcal{O}]=[_{-(D-\{e,f\})}\mathcal{O}]$.
\end{proof}

Now we know that $\mathbb{Z}^E\circlearrowright\mathcal{G}(M)$ is well-defined, so we show next that this  action descends to a group action $\Jac(M)\circlearrowright\mathcal{G}(M)$.

\begin{proposition} \label{Prop:JacActWD}
The stabilizer of the action on any $[\mathcal{O}]$ contains $\Lambda_A(M)\oplus \Lambda_A^*(M)$.
\end{proposition}

\begin{proof}
Let $\overrightarrow{C}\in \Lambda_A(M)$ be a signed circuit. Let $F$ be the set of elements in $C$ whose orientations in $\mathcal{O}$ are the same as in $\overrightarrow{C}$. Then $\overrightarrow{C}\cdot[\mathcal{O}]=(\sum_{\overrightarrow{e}\in\overrightarrow{C}\setminus F}\overrightarrow{e})\cdot[_{-F}\mathcal{O}]=(\sum_{\overrightarrow{C}\setminus F}\overrightarrow{e_i})\cdot[_{-(C-F)}\mathcal{O}]=[\mathcal{O}]$. The proof for $\Lambda_A^*(M)$ is similar.
\end{proof}

\subsection{The action is simply transitive}

\begin{proposition} \label{Prop:TransTorsor}
The group action $\Jac(M)\circlearrowright\mathcal{G}(M)$ is transitive.
\end{proposition}

\begin{proof}
Given any two orientations $\mathcal{O},\mathcal{O}'$, let $\gamma$ be the sum of the (oriented) elements in $\mathcal{O}$ whose orientation in $\mathcal{O}'$ is different, then $[\gamma]\cdot[\mathcal{O}]=[\mathcal{O}']$. 
\end{proof}

By Proposition \ref{Prop:JacActWD} and Proposition \ref{Prop:TransTorsor}, we know that $\Jac(M)\circlearrowright\mathcal{G}(M)$ is well-defined and transitive, and we know that $|\Jac(M)|=|\mathcal{G}(M)|$, so the action is automatically simple.  However, it seems worthwhile to give a direct proof of the simplicity of the action which does not make use of the equality $|\Jac(M)|=|\mathcal{G}(M)|$, since this yields an independent and ``bijective'' proof of the equality.  We begin with the following reduction.

\begin{proposition}
The simplicity of the group action $\Jac(M)\circlearrowright\mathcal{G}(M)$ is equivalent to the statement that every element of the quotient group $\frac{\mathbb{Z}^E}{\Lambda_A(M)\oplus \Lambda_A^*(M)}$ contains a coset representative whose coefficients are all $1,0,-1$.
\end{proposition}

\begin{proof}
Suppose such a set of coset representatives exists. We need to show that whenever $[\gamma]\in\frac{\mathbb{Z}^E}{\Lambda_A(M)\oplus \Lambda_A^*(M)}$ fixes some circuit-cocircuit reversal class, $[\gamma]=[0]$. By transitivity, $[\gamma]$ will fix every equivalence class in such a case. Without loss of generality, the coefficients of $\gamma$ are all $1,0,-1$ with support $F\subset E$. Pick an orientation $\mathcal{O}$ in which the orientation of every element of $F$ agrees with $\gamma$; then $[\mathcal{O}]=[\gamma]\cdot[\mathcal{O}]=[_{-F}\mathcal{O}]$. Therefore $\mathcal{O}\sim\!_{-F}\mathcal{O}$, meaning that $F$ is a disjoint union of positive circuits and cocircuits in $\mathcal{O}$, i.e., $\gamma\in \Lambda_A(M)\oplus \Lambda_A^*(M)$ and $[\gamma]=[0]$. The proof of the other direction is omitted as it is not being used in this paper.
\end{proof}

\begin{proposition} \label{Prop:JacRep}
Every element of $\frac{\mathbb{Z}^E}{\Lambda_A(M)\oplus \Lambda_A^*(M)}$ contains a coset representative whose coefficients are all $1,0,-1$.
\end{proposition}

\begin{proof}
We will show that there is such a representative in $[\gamma]$ for every $\gamma=\sum_{e\in E} c_ee\in \mathbb{Z}^E$ by lexicographic induction on $|\gamma|_\infty:=\max_{e\in E}|c_e|$ and the number of elements $e$ with $|c_e|=|\gamma|_\infty$. The assertion is trivial if $|\gamma|_\infty\leq 1$, so suppose $|\gamma|_\infty>1$. By choosing a suitable reference orientation we may assume that all coefficients of $\gamma$ are non-negative. Pick an element $e$ whose coefficient $c_e$ equals $|\gamma|_\infty$ and pick a positive (co)circuit $C$ containing $e$. By subtracting $\gamma_C:=\sum_{f\in C} f$ from $\gamma$, all positive coefficients $c_f$ with $f\in C$ decrease by 1, while the zero coefficients become $-1$. Hence $|\gamma-\gamma_C|_\infty\leq|\gamma|_\infty$ and the number of elements $f$ with $|c_f|=|\gamma|_\infty$ strictly decreases. By our induction hypothesis, there exists a representative with the desired form in $[\gamma-\gamma_C]=[\gamma]$.
\end{proof}

\begin{corollary} \label{Prop:SimpleTorsor}
The group action $\Jac(M)\circlearrowright\mathcal{G}(M)$ is simple.
\end{corollary}



\subsection{Computability of the group action}\label{computingstuff}

We now show that the simply transitive action of $\Jac(M)$ on ${\mathcal G}(M)$ is efficiently computable.

\begin{proposition} \label{prop:groupactioncomputable}
The action of $\Jac(M)$ on ${\mathcal G}(M)$ can be computed in polynomial time, given a totally unimodular matrix $A$ representing $M$.
\end{proposition}

\begin{proof}
First we show that computing the action of a generator $[\overrightarrow{e}]$ on a circuit-cocircuit reversal class can be done in polynomial time. To see this, note that by Proposition \ref{Prop:Torsor_WellDef}, it suffices to find a positive circuit/cocircuit containing a given element $e$ in $\mathcal{O}$. For positive circuits, this can be done by solving the integer program $\min({\bf 1}^Tv:Av=0, v_e=1, 0\leq v_i\leq 1, v_i\in\mathbb{Z})$ and take the support of the minimizer (if exists), but it is actually a linear program (thus polynomial time computable \cite{schrijver1986LP}) as $A$ is totally unimodular. The cocircuit case is similar.

It remains to show that it is possible to find, in polynomial time, a coset representative with small polynomial-size coefficients in each element of $\Jac(M)\cong\frac{\mathbb{Z}^E}{\Lambda_A(M)\oplus \Lambda_A^*(M)}$.
For the practical reason of generating random elements of $\Jac(M)$ (cf. \S\ref{sec:sam_algo}), we often start with a vector ${\bf y}\in\mathbb{Z}^r$ representing a coset of $\frac{\mathbb{Z}^r}{\Col_{\mathbb{Z}}(AA^T)}$, before lifting ${\bf y}$ to a vector $\gamma\in\mathbb{Z}^E$.
Thus we describe a two-step algorithm to in fact find a representative in $\mathbb{Z}^E$ where all coefficients belong to $\{ -1,0,1 \}$ (the existence of which is guaranteed by Proposition \ref{Prop:JacRep}), starting with an input vector in ${\bf y}\in\mathbb{Z}^r$
.

In step 1, replace ${\bf y}$ by ${\bf y}':={\bf y}-(AA^T)\lfloor (AA^T)^{-1}{\bf y}\rfloor$, where $\lfloor\ \rfloor$ is the coordinate-wise truncation. The new vector represents the same element in $\frac{\mathbb{Z}^r}{\Col_{\mathbb{Z}}(AA^T)}$, and it is equal to $(AA^T)((AA^T)^{-1}{\bf y}-\lfloor (AA^T)^{-1}{\bf y}\rfloor)$. Since $0\leq x-\lfloor x \rfloor<1$ and each coordinate of $AA^T$ is between $-m$ and $m$, the absolute value of each coordinate of ${\bf y}'$ is at most $mr$.  To work in $\frac{\mathbb{Z}^E}{\Lambda_A(M)\oplus \Lambda_A^*(M)}$, we solve the equation $A\gamma={\bf y}'$, which is a simple linear system; since $A$ is totally unimodular, the absolute value of each coefficient of $\gamma$ is at most $mr^2$.

In step 2, starting with an element $\gamma\in \mathbb{Z}^E$ we obtained in step 1. We apply the procedure described in Proposition \ref{Prop:JacRep} with some modification, namely that after choosing a positive (co)circuit $C$ which contains an element $e$ whose coefficient $c_e$ is maximum in $\gamma$, we subtract $\lfloor\frac{c_e}{2}\rfloor\gamma_C$ from $\gamma$. 
No new element with the absolute value of its coefficients being larger than $\lceil\frac{|\gamma|_\infty}{2}\rceil$ is created in each such step, so after every $O(m)$ steps the maximum absolute value of coefficients is halved, and in a total of $O(m\log m)$ steps the maximum absolute value of coefficients is reduced to at most 1. We remark that step 2 by itself can yield a polynomial time algorithm if we work in $\mathbb{Z}^E$ from the beginning.
\end{proof}

\subsection{An algorithm for sampling bases of a regular matroid} \label{sec:sam_algo}

By mimicking the strategy from \cite{baker2013chip}, we can now produce a polynomial-time algorithm for randomly sampling bases of a regular matroid.  The high-level strategy is:
\begin{enumerate}
\item Compute the Smith Normal Form of a matrix $A$ representing $M$, and decompose $\Jac(M)$ as a direct sum of finite abelian groups.
\item Use such a decomposition to choose a random element $\gamma \in {\rm Jac}(M)$.
\item Given a reference orientation ${\mathcal O}$, compute $[{\mathcal O}'] := \gamma \cdot [{\mathcal O}] \in {\mathcal G}(M)$, where $\cdot$ is the group action from Theorem \ref{thm:torsortheorem}.
\item Compute the basis $B$ corresponding to the $(\sigma,\sigma^*)$-compatible orientation in $[\mathcal{O}']$, which can be found in polynomial time by Proposition \ref{prop:CCMalgo}.
\end{enumerate}

\section{Dilations, the Ehrhart polynomial, and the Tutte polynomial}

Metric graphs can either be viewed as limits of subdivisions of discrete graphs or as intrinsic objects. See, for example, Section 2 of  \cite{baker2006metrized}.
Similarly, one can view continuous orientations of regular matroids as a limit of discrete orientations or as intrinsic objects.  So far in this paper we have taken the latter viewpoint, but in this section we shift towards the former.  In doing so, we will see that the bridge between discrete and continuous orientations of regular matroids is intimately related to Ehrhart theory for unimodular zonotopes.  For example, we demonstrate how this perspective allows for a new derivation of a result of Stanley which states that the Ehrhart polynomial of a unimodular zonotope is a specialization of the Tutte polynomial.  Stanley's original proof utilizes a half-open decomposition of a zonotopal tiling.  In contrast, zonotopal tilings will not make an appearance in our proof, although Corollary~\ref{coro:latticepoint_are_vertex} provides a connection to Stanley's argument.

\subsection{The Ehrhart polynomial and the Tutte polynomial}

The Tutte polynomial $T_M(x,y)$ is a bivariate polynomial associated to a matroid $M$ which encodes a wealth of information associated to $M$.  One of its key properties is that  $T_M(x,y)$ is ``universal'' with respect to deletion and contraction, in the following sense:

\begin{proposition}[see~{\cite[Theorem 1]{welsh2000potts}} and~{\cite[Theorem 2.16]{welsh1999tutte}}] \label{thm:gentutte}
Let $\mathbb{M}$ be the set of all matroids. Suppose $a,b,x_0,y_0 \in \RR$ and that $f\colon \mathbb{M} \to \RR$ is a function with $f(\emptyset)=1$ and such that for every matroid $M$ and every element $e$ of $M$,
\begin{align*}
f(M) &= af(M / e) + bf(G\setminus e) &\textrm{if $e$ is neither a loop nor a coloop}\\
f(M) &= x_0f(M \setminus e)  &\textrm{if $e$ is a coloop} \\
f(M) &= y_0f(M / e)  &\textrm{if $e$ is a loop.}\\
\end{align*}
Then  
\[f(M) = a^{rk(M)}b^{rk(M^*)}T_{M}(\frac{x_0}{a},\frac{y_0}{b}).\]
\end{proposition}

Given an integer polytope $P$, its {\it Ehrhart polynomial} $E_P(q)$ counts the number of lattice points in  $qP$, the $q$-th dilate of $P$.  The fact that such a polynomial exists for any integer polytope was proven by Ehrhart \cite{ehrhart1962polynomial}.  Let $M$ be a regular matroid represented by the totally unimodular $r \times m$ matrix $A$.
Given a positive integer $q$, define $qA$ to be the $r \times qm$ matrix obtained by repeating each column of $A$ $q$ times consecutively.
Let $qM$ be the corresponding regular matroid.  
Note that the zonotope $Z_{qA}$ associated to $qA$ is just the $q$-th dilate $qZ_A$.

\medskip

Let $\sigma_q$ be an acyclic signature of $qM$. Using the interpretation of lattice points of $Z_{qA}$ as $\sigma_q$-compatible orientations of $qM$, we give a new proof of the following theorem.
 
\begin{theorem}[Stanley \cite{stanley1991ehrhart}]\label{Ehrhart}
Let $A$ be a totally unimodular matrix with associated zonotope $Z=Z_A$, and let $M$ be the corresponding regular matroid.  
Then
\[
E_Z(q) = q^{{\rm rk}(M)} T_{M}(1 + \frac{1}{q},1).
\]
\end{theorem}

 Stanley's result extends to general integer zonotopes, but {\it a priori} our proof does not.  For a calculation of the Ehrhart polynomial of an integral zonotope using the language of the arithmetic Tutte polynomial, see \cite{moci2012arithmetic}.
Before giving the proof of Theorem~\ref{Ehrhart}, we need a few definitions.  

By a {\em partial orientation} of a regular matroid $M$,\footnote{In the case of graphs, Hopkins and the first author would call such objects ${\it type\, B \, partial\,  orientations}$, but we will suppress the term ``type B" here.} we mean a function $E \to \{-1, 0, 1\}$, where elements mapping to $0$ are called {\em bi-oriented}.    Given $t \in \mathbb{Z}_{>0}$, a {\em $t$-partial orientation of $M$} will be a partial orientation where each bi-oriented edge receives some integer weight $s$ with  $1 \leq s \leq t$. (By convention,  a $0$-partial orientation of $M$ will mean the same thing as an orientation.)  

Fix a reference orientation $\mathcal{O}_{\rm ref}$ on $M$.
Setting $t+1=q$, there is a map from orientations of $qM$ to $t$-partial orientations as follows.  
Given $e \in M$, if all $q$ copies of $e$ are oriented similarly, we map them to the corresponding orientation of $e$ in $M$.  On the other hand, if $s$ copies of $e$ of are oriented in agreement with $O_{\rm ref}$ and $q-s$ copies are oriented oppositely, with $1 \leq s \leq t$, we map this set of edges to a bi-oriented element of weight $s$ in $M$. 

A non-empty subset $F$ of $E$ is called a {\em potential circuit} in a $t$-partial orientation $\mathcal{O}$ if $F$ is a circuit of $M$ and there is a choice of orientation of each bi-oriented element so that $F$ becomes a positive circuit.  We will call a $t$-partial orientation of $M$ {\em circuit connected} if for each $e$ which is the minimum element in a potential circuit, either $e$ is not bi-oriented and is oriented in agreement with the reference orientation, or $e$ is bi-oriented and replacing it with the opposite orientation of $e$ in $\mathcal{O}_{\rm ref}$ does not produce any potential circuits containing $e$.

\begin{proof}
(of Theorem~\ref{Ehrhart})
For each positive integer $q$,
we will define an acyclic signature $\sigma_q$ on $C(qM)$. By Proposition \ref{prop:ctslatticepointprop} and Proposition \ref{prop:discretecompatible}, it will then suffice to prove that the number of $\sigma_q$-compatible orientations of $qM$ is $q^{{\rm rk}(M)} T_{M}(1 + q^{-1},1)$.  As in Example \ref{edgeorder}, each $\sigma_q$ will come from a total order and reference orientation of $qM$. We now explain how given an arbitrary $\sigma_1$, we can naturally define $\sigma_q$.  Given $e \in M$, let $e_1, \dots e_q$ be the $q$ copies of $e$ in $q M$.  We orient the $e_i$ in $\sigma_q$ similarly to $e$ in $\sigma_1$, i.e., so that together they form a positive cocircuit in their induced matroid.  Let $e^i$ be the list of the elements of $M$ according to $\sigma_1$.  Given $e^i_k$ and $e^j_\ell$ in $qM$, we define $\sigma_q$ so that $e^i_k <_q e^j_\ell$ if $i<j$, or $i=j$ and $k<\ell$.

We are attempting to count objects associated to $qM$ using the Tutte polynomial of $M$, so we would first like to produce a bijective map from $\sigma_q$-compatible orientations of $qM$ to certain objects associated to $M$ alone.  To do this, note that given a $\sigma_q$-compatible orientation $\mathcal{O}$ of $qM$, a reference orientation $\mathcal{O}^q_{\rm ref}$, and a set of parallel elements $e_1, \dots e_q$, there are only $q+1$ possible orientations of these elements: $e_1 \dots e_k$ will be oriented in agreement with $\mathcal{O}^q_{\rm ref}$, for some $k=0,1,\ldots,q$, and $e_{k+1} \dots e_q$ will be oriented oppositely.  (If this were not the case, we would have a 2-element positive circuit whose minimum edge is oriented in disagreement with $\mathcal{O}^q_{\rm ref}$.)
Using this observation, it is not difficult to check that the map defined above from orientations of $qM$ to $t$-partial orientations of $M$ (where $t=q-1$) takes $\sigma_q$-compatible orientations of $qM$ bijectively to circuit connected $t$-partial orientations of $M$.

We first prove that the sets $X_{t,M\setminus e}$ and $X_{t,M/e}$ are the images of $X_{t,M}$ under deletion and contraction, respectively.  Given $\mathcal{O} \in X_{t,M\setminus e}$ (the case of $\mathcal{O} \in X_{t,M/ e}$ being similar), suppose that both orientations of $e$ produce $t$-partial orientations which are not elements of $X_{t,M}$.  This implies that both orientations of $e$ produce potential circuits $C_1$ and $C_2$ which are not $\sigma_q$-compatible.  For $1\leq i \leq 2$, we can choose orientations of the bioriented elements in $C_i$ to produce a circuit $C_i'$ which is not $\sigma_q$-compatible.  The sum $C_1'+C_2'$ is in the kernel of $A$ and does not contain $e$, therefore we can apply Lemma \ref{lem:circuitdecomposition} and decompose $C_1'+C_2'$ into a sum of directed circuits not containing $e$ such that the signs of the elements are inherited from $C_1'+C_2'$.  Let $e'$ be the minimum labeled element in $C_1' \cup C_2'$.  It is possible that $e'$ appears in only one of the circuits $C_1'$ or $C_2'$, otherwise it must be oriented similarly in both $C_1'$ and $C_2'$ as they are not $\sigma_q$-compatible. Thus $e'$ is in the support of $C_1'+C_2'$, and there exists a circuit $C_3$ containing $e'$ whose support is contained in the support of $C_1'+C_2'$.  Moreover, $C_3$ has size larger than 2 as $e'$ was oriented similarly in $C_1'$ and $C_2'$, thus it does not correspond to a bioriented element of $\mathcal{O}$.  By assumption, $e'$ is oriented in disagreement with its reference orientation, therefore $C_3$ is not $\sigma_q$-compatible.  After possibly rebiorienting some of the elements in $C_3$, we obtain a potential circuit in $\mathcal{O}$ which is not $\sigma_q$-compatible. This contradicts the assumption that $\mathcal{O} \in X_{t,M\setminus e}$. 

Let $X_{t,M}$ be the set of circuit connected $t$-partial orientations of $M$ (cf.~Figure~\ref{fig:Fig5}).
Let $e$ be the largest element of $M$.  
If $e$ is a loop, then $|X_{t,M}| = |X_{t,M\setminus e}|$, and if $e$ is a coloop, then $|X_{t,M}| = (t+2)|X_{t,M/e}|$.  
If $e$ is neither a bridge nor a loop, we claim that $|X_{t,M}| = |X_{t,M\setminus e}| + (t+1)|X_{t,M/e}|$.  
Given this claim, we conclude from Proposition \ref{thm:gentutte} that $$|X_{t,M}| =(t+1)^{\rm rk(M)}T_{M}(\frac{t+2}{t+1},1) =  q^{{\rm rk}(M)}T_{M}(\frac{q+1}{q},1)$$
as desired.

Take $\mathcal{O} \in X_{t,M}$ and let $\mathcal{O}_e$ be the set of $t$-partial orientations in $X_{t,M}$ which agree with $\mathcal{O}$ away from $e$.  We first observe that $\mathcal{O}_e$ includes a $t$-partial orientation with $e$ bioriented if and only if it includes $t$-partial orientations with $e$ oriented in each direction.  Furthermore, this is the case if and only if $\mathcal{O}/e \in X_{t,M/ e}$.  We always have that $\mathcal{O} \setminus e \in X_{t,M \setminus e}$ as deleting $e$ cannot cause a $t$-partial orientation to stop being circuit connected. Therefore, $|\mathcal{O}_e| =1 $ if and only if $\mathcal{O}/e \notin X_{t,M/e}$, and $|\mathcal{O}_e| = t+2 $ if and only if $\mathcal{O}/e \in X_{t,M/e}$.  The claim now follows by partitioning  $X_{t,M}$ into maximal sets of $t$-partial orientations which agree on every element in $M\setminus e$.  
\end{proof}

\begin{remark}
The realizable part of the Bohne-Dress theorem states that the regular tilings of $Z_A$ by paralleletopes are dual to the generic perturbations of the central hyperplane arrangement defined by $A$. Hopkins and Perkinson \cite{hopkins2016bigraphical} investigated generic bigraphical arrangements, i.e. generic perturbations of twice the graphical arrangement,  and associated certain partial orientations, which they called {\it admissible}, to the regions in the complement of such an arrangement.  The aforementioned duality induces a geometric bijection between these regions and the lattice points in the twice-dilated graphical zonotope.  This in turn gives a bijection between the admissible partial orientations and the circuit connected partial orientations.  The enumeration of these two different classes of partial orientations both appear as specializations of the 
12-variable expansion of the Tutte polynomial from \cite{backman2015fourientation}, and the aforementioned duality interchanges a pair of symmetric variables.
\end{remark}

\subsection{Ehrhart reciprocity}\label{Ehrhartrec}

 Ehrhart reciprocity states that if $P$ is an integral polytope, and $E_P(q)$ is its Ehrhart polynomial, then the number of interior points of the $q$-th dilate of $P$ is $|E_P(-q)|$. Combining Ehrhart reciprocity and Stanley's result, one obtains the following corollary:

\begin{corollary} \label{cor:reciprocity}
The number of interior lattice points in $qZ_A$ is $$q^{{\rm rk}(M)}T_{M}(1-1/q,1).$$
\end{corollary}

\begin{remark} \label{rmk:reciprocity}
 Our proof of Stanley's formula also allows for a direct verification of Corollary~\ref{cor:reciprocity} in this setting without appealing to Ehrhart reciprocity.  Each facet of $qZ$ is determined by a positive cocircuit in $qM$.  Thus a point lies in the interior of $qZ$ if and only if the corresponding $\sigma_q$-compatible orientation of $qM$ contains no positive cocircuits, or equivalently, if every element in the corresponding circuit connected $(q-1)$-partial orientation of $M$ is contained in a potential circuit.  One can verify that these objects are enumerated by the corresponding Tutte polynomial specialization via deletion-contraction as illustrated above, although the argument is slightly more involved as one needs to take care to show that potential circuits and cocircuits can be treated separately.

For the case of graphs, various generalizations of the arguments used in the proof of Theorem~\ref{Ehrhart} are given in \cite{backman2015fourientation}. 
\end{remark}

\subsection{Other invariants of unimodular zonotopes}

The following theorem collects some known connections between evaluations of the Tutte polynomial and geometric quantities associated to unimodular zonotopes.

\begin{theorem}\label{thm:Tutteevals}
Let $Z$ be a unimodular zonotope.  Then:

\begin{itemize}
    \item $T_{M}(2,1)$ is the number of lattice points in $Z$.
    
    \item $T_{M}(0,1)$ is the number of interior lattice points in $Z$.
    
    \item $T_{M}(1,1)$ is the lattice volume of $Z$.
    
    \item $T_{M}(2,0)$ is the number of vertices of $Z$.
    
\end{itemize}

\end{theorem}

\begin{proof}
The first two formulas follow from evaluating the Ehrhart polynomial at $q=1$ and $q=-1$.  The third follows from interpreting the lattice volume of $Z$ as $${\rm Vol}(Z) = \lim_{q \rightarrow \infty} \frac{|\mathbb{Z}^n \cap qZ|}{q^{{\rm rk}(
M)}} = \lim_{q \rightarrow \infty} T_{M}(1+1/q,1) = T_{M}(1,1).$$

The fourth enumeration follows from the classical observation that the normal fan of the zonotope is the central hyperplane arrangement defined by $A$ and then applying Zaslavsky's theorem which says that the number of such regions is $T_{M}(2,0)$.
\end{proof}

\begin{remark}
Recall that $T_M(1,1)$ is equal to the number of bases of $M$, which is equal to $|{\rm Jac}(M)|$. One can show that each maximal cell in our polyhedral decomposition of $Z_A$ has volume 1, which gives an alternate proof of the third evaluation in Theorem~\ref{thm:Tutteevals}.   
Taking the limit of $qZ_{A}$ as $q$ goes to infinity while scaling the lattice by $\frac{1}{q}$, the set $X_{q-1,M}$ approaches the set of $\sigma$-compatible continuous orientations of $M$ and we recover the subdivision from Proposition \ref{prop:zonotopedecomp} (see Figure~\ref{fig:Fig5}).
\end{remark}

\begin{figure}[ht!] \label{fig:Fig5}
\begin{center}
    \includegraphics[width=7cm, height = 7cm, keepaspectratio]{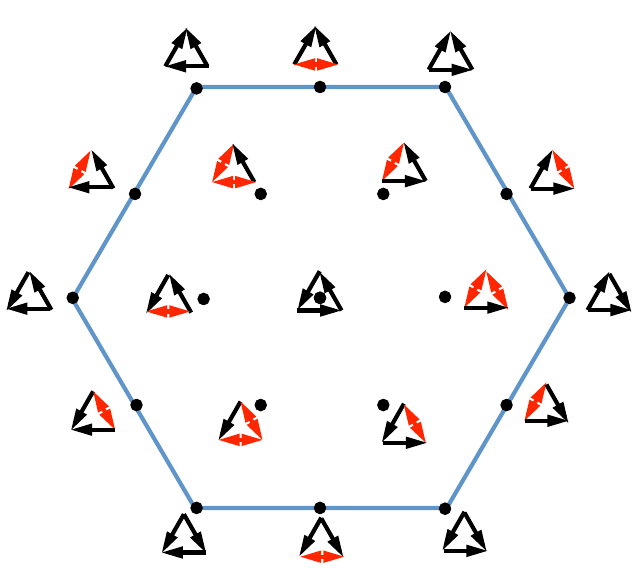}
\end{center}
  \caption{The set $X_{1,K_3}$ associated to the lattice points of $2Z_{K_3}$ using the acyclic signature $\sigma$ from Figure \ref{zonotopeK3_3}.  The bioriented edges are colored red. Taking the (suitably rescaled) limit of $qZ_{K_3}$ as $q$ goes to infinity,  the set $X_{q-1,K_3}$  induces the subdivision depicted in Figure \ref{zonotopeK3_6}.}
\end{figure}

\section*{Acknowledgements}

The first author was partially supported by DFG-Collaborative Research Center, TRR 109 “Discretization in Geometry and Dynamics”, and a Zuckerman STEM Postdoctoral Scholarship; he thanks Sam Hopkins for introducing him to Theorem \ref{thm:Tutteevals}, and Raman Sanyal for explaining that regular tilings of a zonotope can alternately be viewed as dual to generic perturbations of the associated hyperplane arrangement. 
The second author's work was partially supported by the NSF research grant DMS-1529573.
The third author was partially supported by NWO Vici grant 639.033.514; he thanks Yin Tat Lee for the discussion on linear programming.
All three authors thank the anonymous referee for the helpful feedback.

\bibliographystyle{plain}

\bibliography{RM_Geom_Bij}

\medskip

Einstein Institute of Mathematics

The Hebrew University of Jerusalem

Givat Ram. Jerusalem, 9190401, Israel

Email address: \url{spencer.backman@mail.huji.ac.il}\\

School of Mathematics, Georgia Institute of Technology

Atlanta, Georgia 30332-0160, USA

Email address: \url{mbaker@math.gatech.edu}\\

Mathematical Institute, University of Bern

3012 Bern, Switzerland; and

School of Mathematics, Georgia Institute of Technology

Email address: \url{chi.yuen@math.unibe.ch}

\end{document}